\documentclass[
12pt]{amsart}

\usepackage [latin1]{inputenc}
\usepackage{amsmath}
\usepackage{amssymb}
\usepackage{latexsym}
\usepackage{amsfonts}
\usepackage{graphpap}
\usepackage[backrefs]{amsrefs}
\usepackage{bigdelim}
\usepackage{multirow}

\DeclareSymbolFont{SY}{U}{psy}{m}{n}
\DeclareMathSymbol{\emptyset}{\mathord}{SY}{'306}

\theoremstyle{plain}

\newtheorem{thm}{Theorem}[section]
\newtheorem{cor}[thm]{Corollary}
\newtheorem{lem}[thm]{Lemma}
\newtheorem{prop}[thm]{Proposition}
\newtheorem{defn}[thm]{Definition}

\theoremstyle{definition}
\newtheorem{rem}[thm]{Remark}

\newtheorem{3.7}[thm]{}
\newtheorem{7.3}[thm]{}

\numberwithin{equation}{section}
\def\A{{\mathcal A}}
\def\C{{\mathbb C}}

\def\H{{\mathcal H}}
\def\L{\mathcal L}

\def\N{\mathcal{N}}

\def\K{\mathcal K}
\def\B{\mathcal B}
\def\L{\Longrightarrow}

\def\beq{\begin{eqnarray}}
\def\eeq{\end{eqnarray}}
\def\beqa{\begin{eqnarray*}}
\def\eeqa{\end{eqnarray*}}

\def\M{\boldsymbol M}

\def\CCC{\mathbb C}
\def\DDD{\mathbb D}

\newcommand{\U}{\mathcal{U}}

\def\B{\mathcal B}

\def\K{\mathcal K}
\def\U{\mathcal U}

\def\B{\mathcal B}
\def\C{\mathcal C}
\def\H{\mathcal H}

\def\K{\mathcal K}
\def\L{\mathcal L}
\def\M{\mathcal M}

\def\CCC{\mathbb C}
\def\DDD{\mathbb D}

\def\amslatex{$\mathcal{A}\kern-.1667em\lower.5ex\hbox{$M$}\kern-.125em\mathcal{S}$-\LaTeX}

\begin{document}
\title[Cowen-Douglas operators and the third of Halmos' ten problems]{Cowen-Douglas operators and the third of Halmos' ten problems}
\author{Chunlan Jiang,  Junsheng Fang and Kui Ji}
\curraddr{School of Mathematical Sciences, Hebei Normal University, Shijiazhuang, Hebei 050016, China}

\email[C. Jiang]{cljiang@hebtu.edu.cn}
\email[J. Fang]{jfang@hebtu.edu.cn}
\email[K. Ji]{jikui@hebtu.edu.cn, jikuikui@163.com}


\begin{abstract}
Let $T$ be a bounded linear operator on a complex separable infinite dimensional  Hilbert space $\H$.  $T$ is called intransitive if it leaves invariant spaces other than 0 or the whole space $\H$; otherwise it is transitive.
In 1970, P. R. Halmos raised ten open problems on operator theory.  In the past more than 50 years, nine of Halmos' ten problems were answered, but only the third one has made little progress. The third problem of Halmos is  the following: if an intransitive operator has an inverse, is its inverse also intransitive? In this paper, we  establish a set of theoretical systems with the help of Cowen-Douglas operators and spectral analysis. We give an affirmative answer to this problem under certain spectral conditions, which make essential progress in the research of Halmos' third problem.  As the first  application, we show that for an invertible  hyponormal operator $T$, if $T^{-1}$ is intransitive and int$\sigma(T^{-1})^{\land}$ is not connected, then $T$ is also intransitive. As the second application, we show that if $T^{-1}$ has a proper strictly cyclic invariant subspace and there exists a bounded open set $\Omega$ which is a connected component of $\rho(T^{-1})$ such that $\Omega\cap \U_0=\emptyset$, where $\U_0$ is the connected component of $int(\sigma(T^{-1})^\land)$ containing zero point,  then $T$ is intransitive.
\end{abstract}

\maketitle
\section{Introduction}

Let $\mathcal H$ be a complex separable infinite dimensional Hilbert space and $\mathcal{B}({\mathcal H})$ denote the set of bounded linear operators on $\mathcal H$.
An operator $T\in \mathcal{B}({\mathcal H})$ is called intransitive if it leaves invariant spaces other than 0 or the whole space $\H$; otherwise it is transitive.

\textbf{Invariant subspace problem: }Let $T\in \mathcal{B}({\mathcal H})$. Does $T$ have a nontrivial invariant subspace?

It has been more than a century since the invariant subspace problem was raised, and it is still an open question.

A problem closely related to the famous invariant subspace problem is the third problem of Halmos. In 1970, P. R. Halmos raised ten open problems on operator theory. The third problem of Halmos is  the following.

{\bf Problem 3}\,\, {\it If an intransitive operator has an inverse, is its inverse also intransitive?}

The third problem of Halmos can be expressed in the following way: If an invertible operator T has no nontrivial invariant subspaces, does $T^{-1}$ also have no nontrivial invariant subspaces?
Or equivalently, if every nonzero vector $x$ is a cyclic vector of $T$, is  $x$ also a cyclic vector of $T^{-1}$.

A nonzero vector $x$ in $\mathcal{H}$ is called a cyclic vector of $T$ if $span\{T^{k}x\}_{k=0}^\infty=\H$, where $span\{T^{k}x, k\geq0\}$ is the \emph{norm closure} of the linear span of elements in $\{T^{k}x\}_{k=0}^{\infty}$.
In~2007, R. G. Douglas, C. Foias and C. Pearcy showed in \cite{DFP} that
Problem 3 cannot be answered by proving that $T$ and $T^{-1}$ have the same set of cyclic vectors: three kinds of operators are given for which these two sets of cyclic vectors differ. In 2004, the first author discussed Halmos' third problem with R.G.Douglas and came up Lemma~\ref{Theorem 1.1} of this paper, which plays a key role in this paper. The first author thanks R.G.Douglas for many stimulating discussions on Halmos' third problem.

R. G. Douglas, C. Foias and C. Pearcy also pointed out in \cite{DFP}  that the third problem of Halmos is the last unresolved open problem of the ten problems of Halmos. Unfortunately, we only find part of the literature.
In 1972, Problem 1 was solved by P. A. Fillmore, J. G. Stampfli and J. P. Williams in \cite{FSW}.
In 1974, Problem 7 was solved by C. Apostol and  D. Voiculescu  in \cite{Voi}.
In 1976, Problem 8 was solved by D. Voiculescu in \cite{Voi2}.
In 1983, Problem 5 was  resolved by  C.C.Cowen, J.Long and S. Sun and the readers refer to references \cite{CC}, \cite{S1}, \cite{S2}.
In 1997,  Problem 6 was solved by G. Pisier in \cite{Pis}.

For $T\in \B(\H)$, let $\sigma(T)$ denote the spectrum of $T$ and $\rho(T)=\CCC\setminus \sigma(T)$. Let $Lat T$ denote the invariant subspaces lattice of $T$. For a Hilbert space $\K\subseteq \H$, $P_\K$ denote the orthogonal projection from $\H$ onto $\K$.
Let $\sigma(T)^{\land}$ denote the polynomial convex hull of  $\sigma(T)$.  For a compact subset $K$ of $\CCC$, the polynomial convex hull $K^\land$ of $K$
is defined by (~\cite{Sto}, p23)
\[
K^\land\triangleq \left\{z\in\CCC:\, |f(z)|\leq \sup_{w\in K}|f(w)|\,\text{for all polynomials $f$}\right\}.
\]
By~(\cite{Sto}, p24), $\sigma(T)^{\land}$ is the union of $\sigma(T)$ and all of the bounded connected components of $\CCC\setminus \sigma(T)$.

In the past more than 50 years, nine of Halmos' ten problems were answered, but only the third one has made little progress. In this paper, we have established a set of theoretical systems with the help of Cowen-Douglas operators and spectral analysis, which has made essential progress in the research of Halmos' third problem. The specific results are as follows:\\

\textbf{Theorem A. (Theorem~\ref{T:main theorem 1})}
 Suppose $T\in \B(\H)$ is an invertible operator, $x$ is a nonzero noncyclic vector of $T^{-1}$, and $\N(x)=span\{T^{-k}x:\, k\geq 1\}.$ Let $\overline{T}_1(x)=P_{\N(x)} T^{-1} P_{\N(x)}$. If $\sigma(\overline{T}_1(x))^\land\cap \rho_F(\overline{T}_1(x))$ has a connected component which does not contain zero point, then $T$ is intransitive.\\

 \textbf{Theorem B. (Theorem~\ref{T:main theorem 2})}
 Suppose $T\in \B(\H)$ is an invertible operator, $x$ is a nonzero noncyclic vector of $T^{-1}$, and $\N(x)=span\{T^{-k}x:\, k\geq 1\}.$ Let $\overline{T}_2(x)=P_{\N(x)^\perp} T^{-1} P_{\N(x)^\perp}$. If $\sigma(T^{-1})\subseteq \sigma(\overline{T}_2(x))$ and there exists a bounded open set $\Omega$ which is a connected component of $\rho(T^{-1})$ such that $\Omega\cap \U_0=\emptyset$, where $\U_0$ is the connected component of $int(\sigma(T^{-1})^\land)$ containing zero point, then $T$ is intransitive.\\

 As applications of Theorem A and Theorem B, we have the following results.\\

 \textbf{Theorem C. (Theorem~\ref{T:hypernormal})}
 For an invertible  hyponormal operator $T\in\B(\H)$, if $T^{-1}$ is intransitive and int$\sigma(T^{-1})^{\land}$ is not connected, then $T$ is also intransitive.\\

 \textbf{Theorem D. (Theorem~\ref{T:cyclic})}
For an invertible  operator $T\in\B(\H)$, if $T^{-1}$ has a proper strictly cyclic invariant subspace and there exists a bounded open set $\Omega$ which is a connected component of $\rho(T^{-1})$ such that $\Omega\cap \U_0=\emptyset$, where $\U_0$ is the connected component of $int(\sigma(T^{-1})^\land)$ containing zero point,  then $T$ is intransitive.\\

The difficulty in solving the third problem of Halmos is that we only have the information that $T$ is invertible. Therefore, if we are able to answer this problem affirmatively, we need to give the operator $T$ more structure; on the other hand, if the answer to this problem is negative, it means that the answer to the invariant subspace problem is also negative.
In view of this reason, we introduce the Cowen-Douglas operators as a bridge to solve this problem. In this paper, we use  rigid structures of  Cowen-Douglas operators to give  partial answers to this problem.

This paper is divided into seven sections, organized as follows. In Section 2, we introduce some basic knowledge about operator spectrum theory. In Section 3, we introduce some results about Cowen-Douglas operators. In Section 4, we discuss  cyclic vectors. In Section 5, we will study the operator structures of operators $T$ and $T^{-1}$ and build up connections between  $T$, $T^{-1}$ and Cowen-Douglas operators. In Section 6, the proofs of the main theorems are given.
In the last section, we give some applications of the main theorems including the descriptions of invariant subspace of hyponormal operators whose spectrums are not ``thick'' (In~\cite{Brown}, Brown proved that each hypernormal operator with ``thick'' spectrum has a nontrivial invariant subspace) and invertible operators which have proper strictly cyclic invariant subspaces.

We refer to \cite{NKT}, \cite{Brown}, \cite{FJK}, \cite{FJK2}, \cite{Ha1}, \cite{Ha2},  \cite{Her}, \cite{Lom}, \cite{LS2}, \cite{RR2} for more information on Halmos' third problem and the invariant subspace problem. We refer to \cite{Her2}, \cite{JW}, \cite{MU} for operator theory on Hilbert spaces.\\

\noindent{\bf Acknowledgements:} C.Jiang, J.Fang and K.Ji contributed equally to this work. The authors thank Professor Cheng Lixin, Ji Youqing, Wang Kai, Wu jinsong, and Zhang yuanhang for many valuable discussions on the paper.

\section{Preliminaries}
In this paper, $\mathcal H$ is a complex separable infinite dimensional Hilbert space and $\mathcal{B}({\mathcal H})$ is the set of bounded linear operators on $\mathcal H$.
An operator $T\in \mathcal{B}({\mathcal H})$ is called compact if the image of the unit ball of $\mathcal H$ under $T$ is a compact subset of  $\mathcal H$.
Let ${\mathcal K}(\mathcal{H})$ be the ideal of compact operators in $\mathcal{B}({\mathcal H})$, and
$\pi: \mathcal{B}(\mathcal {H})\rightarrow  \mathcal{B}(\mathcal {H})/{\mathcal K}(\mathcal{H})$ (Calkin algebra) be the natural homomorphism projection.  The essential spectrum of $T$, $\sigma_{e}(T)$, is the spectrum of $\pi(T)$ in $ \mathcal{B}(\mathcal {H})/{\mathcal K}(\mathcal{H})$
and $\mathbb{C}\backslash \sigma_{e}(T)$ is called the Fredholm domain of $T$ and is denoted by $\rho_{F}(T)$. Operator $T$ is called semi-Fredholm, if $\mbox{ran}(T)$ (the range of $T$)   is closed, and either dim$\,\ker T$ or dim$\,\ker T^{*}$ is finite. In this case,
$\mbox{ind}\,T=\mbox{dim}\,\ker T-\mbox{dim}\,\ker T^{*}$. A semi-Fredholm operator $T$ is called Fredholm if $\mbox{ind}\,T$ is finite.
By the famous Atkinson theorem, we know $T$ is Fredholm if and only if $\pi(T)$ is invertible in $\mathcal{B}(\mathcal {H})/{\mathcal K}(\mathcal{H})$.

For $T\in \B(\H)$, let $\sigma(T)$ denote the spectrum of $T$ and $\rho(T)=\CCC\setminus \sigma(T)$.  The point spectrum of $T$, $\sigma_{p}(T)$, is defined by $\sigma_{p}(T)=\{\lambda: \mbox{dim}\,\ker (T-\lambda)\geq1\}$. Let $Lat T$ denote the invariant subspaces lattice of $T$.

The definition of strongly irreducible operators was introduced by Gilfeather~\cite{Gil} and Jiang Zejian~\cite{Jia} in 1970's, respectively.

\begin{defn}\label{2.1}
An operator $T\in \mathcal{B}({\mathcal H})$ is called strongly irreducible, if the commutant of $T$, $\mathcal{A}'(T)=\{X\in \mathcal{B}({\mathcal H})|XT=TX\}$, doesn't have any nontrivial idempotent operator. Otherwise, it is called strongly reducible.
\end{defn}
Recall that an operator $P\in \mathcal{B}({\mathcal H})$ is called an idempotent if $P^{2}=P$ and a nontrivial idempotent if $P\neq 0$ or $I$.
\begin{defn}
A nonzero vector $x\in \H$ is called cyclic of $T$ if $span\{T^{k}x,\,k\geq0\}=\H$, where ``span'' denote the norm closure of linear span of the set.
\end{defn}
Let $\mathcal{C}(T)\triangleq\{x|x$ is a cyclic vector of $T\}$.
An operator $T\in \mathcal{B}({\mathcal H})$ is transitive if and only if $\mathcal{C}(T)=\H\backslash\{0\}$.
It is clear that an operator $T\in \mathcal{B}({\mathcal H})$ is strongly reducible, then $T$ is intransitive. When do $T$ and $T^{-1}$ have a common nontrivial invariant subspace? We have the following results.

\begin{prop}\label{1.6}
Let $T\in \mathcal{B}({\mathcal H})$ be an invertible operator. If $T$ is strongly reducible, then so is $T^{-1}$. Furthermore, $T$ and $T^{-1}$ have a common nontrivial invariant subspace.
\end{prop}
\begin{proof}
Suppose that $P\in \mathcal{A}'(T)$ is a nontrivial idempotent operator. Then $P\in \mathcal{A}'(T^{-1})$. Thus $Ran\,P$ and $Ran\,(I-P)$ are both nontrivial invariant subspaces of $T$ and $T^{-1}$.
\end{proof}

\begin{cor}
Let $T\in \mathcal{B}({\mathcal H})$ be an invertible operator. If $\sigma(T)$ is not connected, then $T$ and $T^{-1}$ have a common nontrivial invariant subspace.
\end{cor}
\begin{proof}
Suppose $\sigma(T)$ is not connected.
By the Riesz functional calculus,  there is a nontrivial idempotent $P$ such that $P\in \mathcal{A}'(T)\cap \mathcal{A}'(T^{-1})$.
Therefore,  $RanP\in Lat(T)\cap Lat(T^{-1})$.
\end{proof}

\begin{prop}
Let $T\in \mathcal{B}({\mathcal H})$ be invertible. If $\sigma_{p}(T)\cup\sigma_{p}(T^{*})\neq\emptyset$, then $T$ and $T^{-1}$ have  a common nontrivial invariant subspace.
\end{prop}
\begin{proof}
Without loss of generality, we assume that $T$ is not a scalar operator.
Suppose that $w\in \sigma_p(T)\cup \sigma_p(T^*)$. We have two cases.

(1) If $w\in \sigma_p(T)$, then we can find some $0\neq x \in \mathcal{H}$ such that $(T-w)x=0$. This means
$T^{-1}((T-w)x)=(I-wT^{-1})x=0$. Then $x\in \ker(T^{-1}-w^{-1})$. Thus, $\ker(T-w)\in Lat(T^{-1})\cap Lat(T)$.

(2) If $w\in \sigma_p(T^*)$, then for $x\in \ker(T^{*}-w)$, we have $x\in \ker((T^{*})^{-1}-w^{-1})$. Thus, $$\ker(T^*-w)^{\bot}\in Lat(T^{-1})\cap Lat(T).$$
\end{proof}

Let $\sigma(T)^{\land}$ denote the polynomial convex hull of  $\sigma(T)$. Note that $\sigma(T)^{\land}$ is the union of $\sigma(T)$ and all of the bounded connected components of $\CCC\setminus \sigma(T)$.
\begin{prop}\cite{DFP}
Let $T\in \mathcal{B}({\mathcal H})$ be an invertible operator. If $0\notin \sigma(T)^{\land}$, then $$Lat(T)=Lat(T^{-1})\ \mbox{and}\ \mathcal{C}(T)=\mathcal{C}(T^{-1}).$$
\end{prop}

At the end of this section, we give the following well-known result.

\begin{lem}\label{L:invariant subspace}
Let $T\in \B(\H)$ and $\K$ be an invariant subspace of $T$. Write
\[
T=\begin{pmatrix}
T_1& \quad X\\
0&\quad T_2
\end{pmatrix}
\begin{matrix}
\K\\
\K^\perp
\end{matrix}.
\]
Then $\sigma(T_1)\subseteq\sigma(T)^\land$ and $\sigma(T_2)\subseteq \sigma(T)^\land$.
\end{lem}

\section{Cowen-Douglas operators}
In this section, we introduce basic properties and techniques of Cowen-Douglas operators, which play important role in the proof of our main theorem.
Let $\Omega$ be a bounded open connected subset of the complex plane
$\mathbb{C}$. In \cite{CD}, M. J. Cowen and R. G. Douglas
introduced a class of operators denoted by $B_n(\Omega)$
which contains $\Omega$ as eigenvalues  of
constant multiplicity $n$. The class of Cowen-Douglas operators with
rank $n$, $B_n(\Omega)$ is defined as follows \cite{CD}:
$$\begin{array}{lll}B_n(\Omega):=\{T\in \mathcal{B}(\mathcal{H}):
&(1)\,\,\Omega\subset \sigma(T):=\{w\in \mathbb{C}:T-w ~~
\mbox{is not invertible}\},\\
&(2)\,\,\mbox{span}\{\mbox{ker}(T-w),\,w\in \Omega\}=\mathcal{H},\\
&(3)\,\,\mbox{Ran}(T-w)=\mathcal{H},\\
&(4)\,\,\mbox{dim ker}(T-w)=n, \forall~w\in\Omega.\}
\end{array}$$

We recall the following results in \cite{CD}.
\begin{prop}[1.7.1, \cite{CD}] \label{cd2}Let $T\in B_n(\Omega)$ and $w_0\in \Omega$. Then
$$span\{\mbox{ker}(T-w_0)^{k},\,k\geq1\}=span\{\mbox{ker}(T-w),\,w\in \Omega\}=\mathcal{H}.$$
\end{prop}

\begin{lem}\label{7.1}\cite{CD}
 If $\Omega_{1}\subset\Omega_{2}$, then $B_{1}(\Omega_{2})\subseteq B_{1}(\Omega_{1})$, where $\Omega_{1},\ \Omega_{2}$ are bounded connected open subsets of the complex plane $\mathbb{C}$.
\end{lem}

For Cowen-Douglas operators with index one, we have the following observation.
\begin{lem}\label{3.4}
Let $T\in B_1(\Omega)$. Then  $\sigma_{p}(T^*)=\emptyset$.
\end{lem}
\begin{proof}
Suppose that $e(w)$ is a non-zero vector such that $T(e(w))=we(w)$. By Proposition \ref{cd2}, we know $span\{e(w),\,w\in \Omega\}=\mathcal{H}$. Note that $(T-w_{1})e(w)=(w-w_{1})e(w),\,w_{1}\in \mathbb{C}$.
When $w_{1}\in \Omega$,  we have $Ran(T-w_{1})=\H$ by the definition of $B_1(\Omega)$. If $w_{1}\notin \Omega$, it follows that $w-w_{1}$ is  non-zero.
Thus, $$span\{(T-w_{1})e(w),\,\omega\in \Omega\}=span\{(w-w_{1})e(w),\,\omega\in \Omega\}=span\{e(w),\,\omega\in \Omega\}=\H.$$
This shows that $Ran(T-w_{1})$ is dense in $\H$ and $\ker(T-w_1)^*=\{0\}$. Therefore, $\sigma_{p}(T^*)=\emptyset$.

\end{proof}
\begin{rem}
For any $T\in B_n(\Omega)$, $\sigma_{p}(T^*)=\emptyset,\,n\geq1$.
\end{rem}

\begin{lem}\label{7.2}
For $T\in B_{1}(\Omega)$, there exists a connected open subset $\Phi$ of $\CCC$ such that $\Omega\subseteq\Phi$ and
\begin{enumerate}
\item $B_{1}(\Phi)\subseteq B_{1}(\Omega)$;
\item $\partial\Phi\subset\sigma_{e}(T)$.
\end{enumerate}
In this case, we call $\Phi$ is the maximal domain of $T$.
\end{lem}
\begin{proof}
By Zorn's Lemma, we can find a maximal bounded connected open subset $\Phi$ of $\mathbb{C}$ such that
 $T\in B_{1}(\Phi)$.
Now, we need to prove $\partial\Phi\subset\sigma_{e}(T)$. Otherwise, there exists $\lambda\in\partial\Phi$ such that $\lambda\in\rho_{F}(T)$. That means ind$\,(T-\lambda)=1$ and there exists a neighborhood $O_{\lambda}$ of $\lambda$ such that ind$\,(T-\lambda)=1,\,\,\lambda\in O_{\lambda}$.
By Lemma~\ref{3.4}, $\sigma_{p}((T-\lambda)^{*})=\emptyset$. We have dim$\,\ker(T-\lambda)=1$.
Let $\Phi''=\Phi\cup O_{\lambda}$. Then $\Phi''$ is connected. Thus, we have $T\in B_{1}(\Phi'')$. This contradicts to the maximality of $\Phi$.
\end{proof}

\begin{rem}If $\Omega$ is the maximal domain of Cowen-Douglas operator $T$, then $\partial \Omega\subseteq \sigma_{e}(T).$

\end{rem}

For two operators $A,B\in\B(\H)$, we call $A\sim_s B$ if there exists an invertible operator $X\in \B(\H)$ such that $X^{-1}AX=B$.

\begin{lem}\label{3.3}
Suppose that $A\in \mathcal{B}({\mathcal H})$, $A_{1}\in {\mathcal B}({\mathcal H}\oplus \mathbb{C}e_0)$, and
$$A_{1}=\begin{pmatrix}
0&\quad \alpha_{1}&\quad \alpha_{2}&\quad \alpha_{3}&\quad\cdots\\
0&&&&\\
0&&A&&\\
\vdots&&&&
\end{pmatrix}\begin{matrix}e_{0}&\\
&\\
\H&\\
&\end{matrix}.$$ If $A_{1}$ is surjective and $\mbox{dimker}A_1=1$, then $A\sim_{s}A_{1}$. In particular, if $A_1\in B_1(\Omega)$, $0\in \Omega$ and $A_1e_0=0$, then $
A\sim_{s}A_{1}$.
\end{lem}
\begin{proof}
Since $\mbox{dimker}A_1=1$, we have $\ker A_{1}=\CCC e_{0}$.  Thus, $(\ker A_{1})^{\bot}=\mathcal{H}$.
Let $X=A_{1}|_{(\ker A_{1})^{\bot}}$. Then  $X:\,\mathcal{H}=(\ker A_{1})^{\bot}\rightarrow \mathcal{H}\oplus \mathbb{C}e_0$ is an invertible bounded linear operator.

By the next two equations,
$$\begin{array}{lll}A_{1}X&=&A_{1}\cdot A_{1}|_{(\ker A_{1})^{\bot}}\\
&=&A_{1}(P_{\ker A_{1}}A_{1}|_{(\ker A_{1})^{\bot}}+P_{(\ker A_{1})^{\bot}}A_{1}|_{(\ker A_{1})^{\bot}})\\
&=&A_{1}P_{(\ker A_{1})^{\bot}}A_{1}|_{(\ker A_{1})^{\bot}}\end{array}$$
and
$$\begin{array}{lll}XA&=&A_{1}|_{(\ker A_{1})^{\bot}}\cdot A\\
&=&A_{1}|_{(\ker A_{1})^{\bot}}P_{(\ker A_{1})^{\bot}}A_{1}|_{(\ker A_{1})^{\bot}},
\end{array}$$
we obtain $A_{1}X=XA$, that is, $A\sim_{s}A_{1}$.
\end{proof}

For more on Cowen-Douglas operators, we refer to \cite{CFJ}, \cite{HJ}, \cite{JJ}, \cite{JW}.

\section{Cyclic vectors}

In this section $\Omega$ is a connected open subset of $\CCC$.

\begin{lem}[\cite{Lin}]\label{3.2}
Suppose that $T\in B_n(\Omega)$. Then $\mathcal{C}(T)\neq\emptyset$, where $\mathcal{C}(T)$ is the set of cyclic vectors of $T$.
\end{lem}

\begin{lem}\label{L:cyclic vector}
Suppose that $A\in B_1(\Omega)$, $0\in\Omega$, and $AB=I$. If  $0\neq e_{0}\in \ker A$, then $e_{0}\in \mathcal{C}(B)$. In particular, $\C(B)\neq \emptyset$.
\end{lem}
\begin{proof}
Since $AB=I$, $$\ker A^k=\{e_{0},Be_{0},\cdots, B^{(k-1)}e_{0}\}.$$
Since $A\in B_1(\Omega)$,  $$span\{\ker A^{k},k\geq 0\}=\H.$$  It follows that
$$span\{ B^{k}e_{0},k\geq0\}=\H.$$Thus, we have $e_{0}\in \mathcal{C}(B)$.
\end{proof}

The following result is due to P. A. Fillmore, J. G. Stampfli, J. P. Williams~\cite{FSW}. For the convenience of readers, we give a simple proof.
\begin{prop}\cite{FSW}\label{4.4}
Let $T\in \mathcal{B}(\H)$. If there is a $\lambda\in \mathbb{C}$ such that $dim\,\ker(T-\lambda)^{*}\geq2$, then $\mathcal{C}(T)=\emptyset$.
\end{prop}
\begin{proof}
Without loss of generality, we assume $dim\,\ker(T-\lambda)^{*}=2$.  Suppose that $\{f_{1},f_{2}\}$ is an ONB of $\ker(T-\lambda)^{*}$. Note that $$(\ker(T-\lambda)^{*})^{\bot}=\overline{Ran(T-\lambda)}\in Lat(T).$$
 Let $$T_{1}=T|_{\overline{Ran(T-\lambda)}},\,\,T_{2}=P_{\ker(T-\lambda)^{*}}T|_{\ker(T-\lambda)^{*}},T_{12}=P_{\overline{Ran(T-\lambda)}}T|_{\ker(T-\lambda)^{*}}.$$ We have
$$T_{2}=\begin{pmatrix}
 \,\lambda & \quad 0 \\
 \,0 &\quad \lambda\\
\end{pmatrix},\ T=\begin{pmatrix}
 T_{1} &\quad  T_{12} \\
 0 &\quad  T_{2}\\
\end{pmatrix}\begin{matrix}\overline{Ran(T-\lambda)}&\\
\ker(T-\lambda)^{*}&\end{matrix},\ T^{k}=\begin{pmatrix}
 T_{1}^{k} & \quad * \\
 0 &\quad  T_{2}^{k}\\
\end{pmatrix}\begin{matrix}\overline{Ran(T-\lambda)}&\\
\ker(T-\lambda)^{*}&\end{matrix}.$$
If $\mathcal{C}(T)\neq\emptyset$, then we can find some non-zero vector $y\in \mathcal{C}(T)$.  There exist $\alpha_{1},\,\alpha_{2}\in \mathbb{C}$ such that $$P_{\ker(T-\lambda)^{*}}y=\alpha_{1}f_{1}+\alpha_{2}f_{2}\neq0.$$
Notice that  $span\{T^{k}y,\,k\geq0\}=\H$. On the other hand,
$$\begin{array}{lll}
span\{T_{2}^{k}P_{\ker(T-\lambda)^{*}}y,\,k\geq0\}&=&span\{\lambda^{k}(\alpha_{1}f_{1}+\alpha_{2}f_{2}),\,k\geq0\}\\
&\subseteq&\mathbb{C}\{\alpha_{1}f_{1}+\alpha_{2}f_{2}\}\neq\ker(T-\lambda)^{*}.
\end{array}$$
This contradicts to the fact that $y\in \mathcal{C}(T)$.
\end{proof}

The following proposition is due to D. A. Herrero.
\begin{prop}\cite{Her}\label{4.2}
Let $T\in \mathcal{B}(\H)$. If $\mathcal{C}(T)\neq\emptyset$, then $\rho_{s-F}^{-1}(T)$ is simply connected.
\end{prop}

\begin{prop}\label{4.3}
 Suppose $T\in \mathcal{B}(\H)$, $T^{*}\in B_{1}(\Omega)$,  and $\Omega$ is the maximal domain of $T^{*}$. If $\mathcal{C}(T)\neq\emptyset$, then $\Omega=
  \rho_{s-F}^{-1}(T)$ and $\rho_{s-F}(T)=\rho_{F}(T)$. Furthermore, assume $\Omega'$ is a  bounded connected component  of $\rho_{F}(T^*)=\rho_{s-F}(T^*)$ such that $\Omega'\cap \Omega=\emptyset$. Then $\Omega'\subset \rho(T^*)$.
\end{prop}
\begin{proof}
Clearly, we have $\Omega\subseteq \rho_{s-F}^{-1}(T)$. Since $\mathcal{C}(T)\neq \emptyset$, $\rho_{s-F}^{-1}(T)$ is simply connected by Proposition~\ref{4.2}. Therefore, $\Omega=\rho_{s-F}^{-1}(T)$. Since $\mathcal{C}(T)\neq \emptyset$, $\rho_{s-F}(T)=\rho_{F}(T)$ is a corollary of Lemma~\ref{3.4} and Proposition~\ref{4.4}.

 By Lemma~\ref{3.4}, $\sigma_p(T)=\emptyset$. Since $\Omega^{\prime}$ is a bounded connected component of $\rho_{F}(T)=\rho_{s-F}(T)$ and  $\Omega'\cap \Omega=\emptyset$,   By Proposition \ref{4.4} and Proposition \ref{4.2},  we have that
$$\mbox{dim} \mbox{ker}(T^*-\lambda)<1, \lambda\in \Omega^{\prime}.$$
This means $\mbox{dimker}(T^*-\lambda)=0, \lambda \in \Omega^{\prime}.$ Hence, $\Omega^{\prime}\subset \rho(T^*).$
\end{proof}

\section{Spectral structures of $T$ and $T^{-1}$}
In this section, we first introduce a $2\times 2$ matrix technique to obtain the spectral structures of $T$ and $T^{-1}$.

Let  $\{e_k\}_{k=0}^\infty$ be an orthogonal normal basis (denoted by ``ONB") of a Hilbert space $\H$ and let
$S_1^*$ be the backward shift operator defined as $S_1^*(e_0)=0$, $S_1^*(e_{k+1})=e_k$, $k=0,1,2,\cdots$. For $T\in \B(\H)$ and $x\in \H$, $x\neq 0$, define an operator $T_x\in \B(\H\oplus \H)$ as the following
\[
T_x=\left(\begin{matrix}
T&\quad x\otimes e_0\\
0&\quad S_1^*
\end{matrix}\right).
\]

The following lemma plays a key role throughout this section.
\begin{lem}\label{Theorem 1.1}
Let $T\in\B(\H)$  with spectral radius $r(T)<1$  and let \[T_x=\left(\begin{matrix}
T&\quad x\otimes e_0\\
0&\quad S_1^*
\end{matrix}\right).\] Let $\Sigma$ be the connected component of $\DDD\setminus \sigma(T)$ which contains $\{w\in \DDD: r(T)<|w|<1\}$.  Then we have the following:
\begin{enumerate}
\item For every $w\in \DDD\setminus \sigma(T)$, $dimker(T_x-w)=1$ and $Ran(T_x-w)=\H\oplus \H$;
\item $T_x\in B_1(\Sigma)$ if and only if $x$ is a cyclic vector of $T$, i.e., $span\{T^nx:\,n\geq 0\}=\H$.
\end{enumerate}
\end{lem}
\begin{proof}
(1). Let $w\in \DDD\setminus \sigma(T)$ and $e(w)\triangleq\sum\limits_{k=0}^\infty w^ke_k$. Then $e(w)\in ker(S_1^*-w)$, $w\in \DDD$. We calculate
\[
\left(\begin{matrix}
T-w&\quad x\otimes e_0\\
0&\quad S_1^*-w
\end{matrix}\right)\left(\begin{matrix}
\xi\\
\eta
\end{matrix}\right)=\left(\begin{matrix}
0\\
0
\end{matrix}\right).
\]
This is equivalent to $(S_1^*-w)\eta=0$ and $(T-w)\xi=-\langle \eta,e_0\rangle x$. Since $dimker(S_1^*-w)=1$, there is a $w_1\in\CCC$ such that $\eta=w_1 e(w)$. Without loss of generality, we assume that $w_1=1$.  Note that $\langle e(w), e_0\rangle=1$ and  $T-w$ is invertible for $w\in \DDD\setminus \sigma(T)$. We have $\xi=-(T-w)^{-1}x$. This proves that  $dimker(T_x-w)=1$.

To show that $T_x-w$ is surjective for $w\in \DDD\setminus \sigma(T)$, we need to find for every $\xi'\oplus\eta'\in \H\oplus \H$ a vector $\xi\oplus \eta\in \H\oplus \H$ satisfying the following equation:
\[
\left(\begin{matrix}
T-w&\quad x\otimes e_0\\
0&\quad S_1^*-w
\end{matrix}\right)\left(\begin{matrix}
\xi\\
\eta
\end{matrix}\right)=\left(\begin{matrix}
\xi'\\
\eta'
\end{matrix}\right).
\]

This is equivalent to $(S_1^*-w)\eta=\eta'$ and $(T-w)\xi=\xi'-\langle \eta,e_0\rangle x$. Note that both $S_1^*-w$ and $T-w$ are surjective for $w\in \DDD\setminus \sigma(T)$. The existence of $\xi$ and $\eta$ is clear.

(2). For $w\in \Sigma$, define $y(w)=-(T-w)^{-1}x$. Then by (1), $y(w)\oplus e(w)$ is in $ker(T_x-w)$. To show $T_x\in B_1(\Sigma)$, we  need to prove that
\[
span\{y(w)\oplus e(w), w\in \Sigma\}=\H\oplus \H.
\]

Suppose that there exists an $x_1\oplus x_2\in \H\oplus \H$ such that $\langle y(w)\oplus e(w), x_1\oplus x_2\rangle=0$. Then $\langle y(w),x_1\rangle+ \langle e(w),x_2\rangle=0$.
Since $\langle e(w),x_2\rangle$ is analytic on $\DDD$ and $(w-T)^{-1}=\frac{1}{w}\sum_{n=0}^\infty \left(\frac{T}{w}\right)^n$ for $|w|>r(T)$, $\langle -(T-w)^{-1}x,x_1\rangle$ is analytic when $|w|>r(T)$. Furthermore,
\[
\langle y(w),x_1\rangle=-\langle e(w),x_2\rangle,\quad r(T)<|w|<1.
\]
Thus, by the analytic continuation theorem, we know that $\langle y(w), x_1\rangle$ is analytic on $\CCC$. Since \[\lim_{|w|\rightarrow +\infty}\langle y(w),x_1\rangle =0,\] we see that $\langle y(w),x_1\rangle$ is a bounded entire function on $\CCC$. Thus we have
\[
\langle y(w),x_1\rangle=\langle e(w),x_2\rangle =0.
\]
Note that $\langle (w-T)^{-1}x, x_1\rangle=0$ for all $|w|>r(T)$, we have
\[
\left\langle \sum_{n=0}^\infty \left(\frac{T^nx}{w^{n+1}}\right), x_1\right\rangle=\sum_{n=0}^\infty \langle T^nx,x_1\rangle \frac{1}{w^{n+1}}=0,\quad |w|>r(T).
\]
It follows that $\langle T^n x, x_1\rangle=0$ for $n=0,1,\cdots$. Suppose $x$ is a cyclic vector of $T$. Then $x_1=0$. Since $S_1^*\in B_1(\DDD)\subset B_1(\Sigma)$, we have $span\{e(w): w\in \Sigma\}=\H$. This means $x_2=0$. Thus $span\{y(w)\oplus e(w):\,w\in \Sigma\}=\H\oplus \H$. Suppose $x$ is not a cyclic vector of $T$. Let $0\neq x_1\perp\{T^nx:\, n\geq 0\}$. Then $(x_1\oplus 0)\perp span\{y(w)\oplus e(w):\,w\in \Sigma\}$ and therefore $span\{y(w)\oplus e(w):\,w\in \Sigma\}\neq \H\oplus\H$. This implies that $T_x\notin B_1(\Sigma)$.

\end{proof}

\begin{lem}\label{ri}For every $x\in \mathcal{H}$, $x\neq 0$, define $S\in  \mathcal{B}(\mathcal{H}\oplus \mathcal{H})$ as follows
$$S=\left(\begin{matrix}T^{-1}&\quad 0\\
0&\quad S_1\\ \end{matrix}\right).$$ Then $S_x\triangleq S$ is a right inverse of $T_x$.
\end{lem}

\begin{proof} Notice that $S_1^*(e_0)=0$.  Then we have $(x\otimes e_0)S_1=x\otimes S^*_1(e_0)=0,$ and
\[
T_xS_x=\left(\begin{matrix}T&\quad x\otimes e_0\\
0&\quad S_1^*\\ \end{matrix}\right)\left(\begin{matrix}T^{-1}&\quad 0\\
0&\quad S_1\\ \end{matrix}\right)
=\left(\begin{matrix}I\,&\quad (x\otimes e_0)S_1\\
0\,&\quad I\\ \end{matrix}\right)
=\begin{pmatrix}\quad I\quad &\, 0\quad\\
\quad 0\quad&\, I\quad\\ \end{pmatrix}.
\]

\end{proof}

Define \[
\M(x)=span\{(-T^{-(n+1)}x)\oplus e_n: n\geq 0\}.\]
Then
\[
T_x\left(\begin{matrix}
-T^{-1}x\\
e_0
\end{matrix}\right)=\left(\begin{matrix}
0\\
0
\end{matrix}\right), T_x\left(\begin{matrix}
-T^{-(n+1)}x\\
e_n
\end{matrix}\right)=\left(\begin{matrix}
-T^{-n}x\\
e_{n-1}
\end{matrix}\right)
\] for $n\geq 1$ and
\[
S_x\left(\begin{matrix}
-T^{-(n+1)}x\\
e_n
\end{matrix}\right)=\left(\begin{matrix}
-T^{-(n+2)}x\\
e_{n+1}
\end{matrix}\right)
\] for all $n\geq 0$. Thus $\M(x)\in Lat T_x\cap Lat S_x$.

We set

\begin{equation}
T_x=\left(\begin{matrix}
T&\quad x\otimes e_0\\
0&\quad S_1^*
\end{matrix}\right)\begin{matrix}
\H\\
\H
\end{matrix}=\left(\begin{matrix}
\hat{T}_x&\quad T_{1,2}\\
0&\quad T_2
\end{matrix}\right)\begin{matrix}
\M(x)\\
\M(x)^\perp
\end{matrix},\end{equation} and  \begin{equation} S_x=\left(\begin{matrix}
T^{-1}&\quad 0\\
0&\quad S_1
\end{matrix}\right)\begin{matrix}
\H\\
\H
\end{matrix}=\left(\begin{matrix}
\hat{S}_x&\quad S_{1,2}\\
0&\quad S_2
\end{matrix}\right)\begin{matrix}
\M(x)\\
\M(x)^\perp
\end{matrix}.\end{equation}

\begin{lem}\label{start} Let $a_n=P_{\M(x)}(0\oplus e_n)$ and $b_n=P_{\M(x)^\perp}(0\oplus e_n)$ for $n\geq 0$. Then we have
$$T_2b_n=b_{n-1}; \, \hat{S}_x^*a_n=a_{n-1}, n\geq 1, \, \hat{S}_x^*a_0=0,$$
and
$$\hat{T}^*_xa_n=a_{n+1};\, S_2b_n=b_{n+1}, n\geq 0.$$
\end{lem}

\begin{proof}

Since
\[
T_x\left(\begin{matrix}
0\\
e_n
\end{matrix}\right)=\left(\begin{matrix}
T&\quad x\otimes e_0\\
0&\quad S_1^*
\end{matrix}\right)
\left(\begin{matrix}
0\\
e_n
\end{matrix}\right)
=
\left(\begin{matrix}
0\\
e_{n-1}
\end{matrix}\right),\,\forall n\geq 1,
\] we have
\[
\left(\begin{matrix}
\hat{T}_x&\quad T_{12}\\
0&\quad T_2
\end{matrix}
\right)\left(\begin{matrix}
a_n\\
b_n
\end{matrix}\right)=
\left(\begin{matrix}
a_{n-1}\\
b_{n-1}
\end{matrix}
\right) ,\,\forall n\geq 1.
\]
Thus $T_2b_n=b_{n-1}$, $\forall n\geq 1$. A similar calculation shows that  $T_2b_0=P_{\M(x)^\perp}(x\oplus 0)$.

Since \[
T_x^*\left(\begin{matrix}
0\\
e_n
\end{matrix}\right)=\left(\begin{matrix}
T^*&\quad 0\\
e_0\otimes x&\quad S_1
\end{matrix}\right)
\left(\begin{matrix}
0\\
e_n
\end{matrix}\right)
=
\left(\begin{matrix}
0\\
e_{n+1}
\end{matrix}\right),\,\forall n\geq 0,
\] we have
\[
\left(\begin{matrix}
\hat{T}^*_x&\quad 0\\
T_{12}^*&\quad T^*_2
\end{matrix}
\right)\left(\begin{matrix}
a_n\\
b_n
\end{matrix}\right)=
\left(\begin{matrix}
a_{n+1}\\
b_{n+1}
\end{matrix}
\right),\,\forall n\geq 0.
\]
Thus $\hat{T}^*_xa_n=a_{n+1}$, $\forall n\geq 0$.

Since
\[
S_x\left(\begin{matrix}
0\\
e_n
\end{matrix}\right)=\left(\begin{matrix}
T^{-1}&\quad 0\\
0&\quad S_1
\end{matrix}\right)
\left(\begin{matrix}
0\\
e_n
\end{matrix}\right)
=
\left(\begin{matrix}
0\\
e_{n+1}
\end{matrix}\right), \,\forall n\geq 0,
\] we have
\[
\left(\begin{matrix}
\hat{S}_x&\quad S_{12}\\
0&\quad S_2
\end{matrix}
\right)\left(\begin{matrix}
a_n\\
b_n
\end{matrix}\right)=
\left(\begin{matrix}
a_{n+1}\\
b_{n+1}
\end{matrix}
\right),\,\forall n\geq 0.
\]
Thus $S_2b_n=b_{n+1}$, $\forall n\geq 0$.

Since
\[
S^*_x\left(\begin{matrix}
0\\
e_n
\end{matrix}\right)=\left(\begin{matrix}
\left(T^{-1}\right)^*&\quad 0\\
0&\quad S^*_1
\end{matrix}\right)
\left(\begin{matrix}
0\\
e_n
\end{matrix}\right)
=
\left(\begin{matrix}
0\\
e_{n-1}
\end{matrix}\right), \,\forall n\geq 0,
\] where we define $e_{-1}=0$, we have
\[
\left(\begin{matrix}
\hat{S}_x^*&\quad 0\\
S_{12}^*&\quad S_2^*
\end{matrix}
\right)\left(\begin{matrix}
a_n\\
b_n
\end{matrix}\right)=
\left(\begin{matrix}
a_{n-1}\\
b_{n-1}
\end{matrix}
\right),\,\forall n\geq 0.
\]
Thus $\hat{S}_x^*a_n=a_{n-1}$, $\forall n\geq 1$ and $\hat{S}_x^*a_0=0$.
\end{proof}

Note that
\[\langle a_n,-T^{-(m+1)}x\oplus e_m\rangle=\langle P_{\M(x)}(0\oplus e_n),-T^{-(m+1)}x\oplus e_m\rangle =\langle 0\oplus e_n,-T^{-(m+1)}x\oplus e_m\rangle=\delta_{n,m}.\]
 In particular,
$a_n\neq 0$ for all $n\geq 0$.

\begin{lem}\label{Lemma 1.4}
Let $T\in\B(\H)$, $w_0\in \mathbb{C}$ and $\delta>0$. Suppose $\forall w$, $|w-w_0|<\delta$, and
\begin{enumerate}
\item $Ran(T-w)=\H$;
\item $dimker(T-w)=1$.
\end{enumerate}
Then $span\{ ker(T-w_0)^n,\,n\geq 1\}=span\{ ker(T-w),\, |w-w_0|<\delta\}$.
\end{lem}
\begin{proof}
We may assume that $w_0=0$. Then $Ran T=\H$ and $dimker T=1$. Assume that $T\xi=0$ for some $\xi\neq 0$. Let $\K=span\{ ker(T-w),\, |w|<\delta\}$. Then $\K\in Lat T$ and $\xi\in\K$. Write
\[
T=\left(\begin{matrix}
T_1&\quad T_{12}\\
0&\quad T_2
\end{matrix}
\right)\begin{matrix}
\K\\
\K^\perp
\end{matrix}.
\]
Claim $T_2$ is injective. Otherwise $T_2\eta=0$ for some $\eta\neq 0$. Let $e(w)$ be a nonzero vector such that $(T-w)e(w)=0$. Then $Ran T_1$ contains $span\{ we(w):\,{|w|<\delta}\}$ and  $Ran T_1$ is dense in $\K$. Therefore, there exists a sequence of vectors $\xi_n\in \K$ such that
\[
\lim\limits_{n\rightarrow +\infty} T\left(\begin{matrix}
\xi_n\\
\eta
\end{matrix}\right)=\left(\begin{matrix}
0\\
0
\end{matrix}\right).
\]
Write
\[
\left(\begin{matrix}
\xi_n\\
\eta
\end{matrix}\right)=\left(\begin{matrix}
\xi_n-\alpha_n\xi\\
\eta
\end{matrix}\right)+\alpha_n\left(\begin{matrix}
\xi\\
0
\end{matrix}\right)
\]
as an orthogonal decomposition. Note that $T$ is an invertible map from $\left[\xi\oplus 0\right]^\perp$ onto $Ran T$. Therefore, there exists a $K>0$ such that
\[
\|T\zeta\|\geq K\|\zeta\|,\quad \forall \zeta\in\left[\left(\begin{matrix}
\xi\\
0
\end{matrix}\right)\right]^\perp.
\]

So
$$\lim\limits_{n\rightarrow\infty}T\left(\begin{matrix}
\xi_n-\alpha_n\xi\\
\eta
\end{matrix}\right)=\left(\begin{matrix}
0\\
0
\end{matrix}\right)$$
implies that
\[0=\lim\limits_{n\rightarrow +\infty}\left\|T\left(\begin{matrix}
\xi_n-\alpha_n\xi\\
\eta
\end{matrix}\right)\right\|\geq\limsup_{n\rightarrow\infty}K\left\|\left(\begin{matrix}
\xi_n-\alpha_n\xi\\
\eta
\end{matrix}\right)\right\|\geq K\|\eta\|.\]
Therefore, $\eta=0$. This is a contradiction and thus $T_2$ is injective.

Since $RanT=\H$, there is an operator $$S=\left(\begin{matrix} S_{11}&\quad S_{12}\\
S_{21}&\quad S_{22}\end{matrix}\right)\begin{matrix} \K\\ \K^\perp\end{matrix}$$ satisfying
\[
TS=\left(\begin{matrix}
T_1&\quad T_{12}\\
0&\quad T_2
\end{matrix}
\right)\left(\begin{matrix} S_{11}&\quad S_{12}\\
S_{21}&\quad S_{22}\end{matrix}\right)=1.
\]
This implies that $T_{2}S_{21}=0$ and $T_2S_{22}=1$. Since $T_2$ is injective, $S_{21}=0$ and $T_2$ is invertible.
Then $T_1S_{11}=1$. So $Ran T_1=\K$. Similarly, $Ran(T_1-w)=\K$ for all $|w|<\delta$. Therefore, $T_1\in B_1(\Omega)$.

Note that for $n\geq 1$,
\[
\left(\begin{matrix}
T_1&\quad T_{12}\\
0&\quad T_2
\end{matrix}
\right)^n=\left(\begin{matrix}
T_1^n&\quad\ast\\
0&\quad T_2^n
\end{matrix}
\right).
\]
So
\[
\left(\begin{matrix}
T_1^n&\quad \ast\\
0&\quad T_2^n
\end{matrix}
\right)\left(\begin{matrix}
\xi\\
\eta
\end{matrix}
\right)=\left(\begin{matrix}
0\\
0
\end{matrix}
\right)
\]
implies that $T_2^n\eta=0$ and therefore $\eta=0$. Furthermore, $T_1^n\xi=0$ and $ker T^n=ker T_1^n$. By Proposition~\ref{cd2},  $$span\{ ker(T-w_0)^n,\,n\geq 1\}=span\{ ker(T-w),\, |w-w_0|<\delta\}=\K.$$
\end{proof}

\begin{lem}\label{Theorem 1.5}
Let $T\in\B(\H)$ be invertible, $r(T)<1$, and \[T_x=\left(\begin{matrix}
T& \quad x\otimes e_0\\
0& \quad S_1^*
\end{matrix}\right)\begin{matrix}
\H\\
\H
\end{matrix}=\left(\begin{matrix}
\hat{T}_x&\quad T_{1,2}\\
0& \quad T_2
\end{matrix}\right)\begin{matrix}
\M(x)\\
\M(x)^\perp
\end{matrix}.\] Let $\Phi_0$ be the connected component of $\mathbb{D}\setminus\sigma(T)$ with $0\in \Phi_0$. Let $\Sigma$ be the connected component of $\DDD\setminus \sigma(T)$ which contains $\{w\in \DDD: r(T)<|w|<1\}$.  If $x$ is a cyclic vector of $T$ and $\M(x)\neq \H\oplus\H$, then
\begin{enumerate}
\item $\hat{T}_x\in B_1(\Phi_0)$ and $ \sigma(\hat{T}_x)\subset \sigma(T)^{\land}$ ;
\item $T_2$ is invertible and $T_2\in B_1(\Sigma)$.
\end{enumerate}
\end{lem}

\begin{proof}
(1). By (1) of Lemma~\ref{Theorem 1.1}, $\forall w \in \Phi$, $Ran(T_x-w)=\H\oplus \H=\M(x)\oplus \M(x)^\perp$ and $dimker(T_x-w)=1$. By Lemma~\ref{Lemma 1.4},
\[
\M(x)=span\{-T^{-(n+1)}x\oplus e_n: n\geq 0\}=span\{ ker T_x^n:\,n\geq 1\}=span\{ ker(T_x-w):\, {w\in \Phi_0}\}.
\]
 Therefore, $ker (T_x-w)\subset \M(x)$ for $w\in \Phi_0$. So $ker(\hat{T}_x-w)=ker(T_x-w)=\mathbb{C} f(w)\subset \M(x)$, where $f(w)=y(w)\oplus e(w)$ as in the proof of Lemma~\ref{Theorem 1.1}. Since
\[
(\hat{T}_x-w_1)f(w)=(w-w_1)f(w),
\]
$Ran(\hat{T}_x-w_1)$ is dense in $\M(x)$, $w_1\in \Phi_0$.

Suppose $Ran(\hat{T}_x-w)\neq \M(x)$ for some $w\in \Phi_0$. Then there exists a $y\in \M(x)$ such that $y\notin Ran(\hat{T}_x-w)$ and $y\oplus 0\in Ran(T_x-w)$. Therefore, there exists a vector $\xi\oplus \eta\in \M(x)\oplus \M(x)^\perp$ such that $\eta\neq 0$ and
\[
(T_x-w)\left(\begin{matrix}
\xi\\
\eta\end{matrix}\right)=\left(\begin{matrix}
y\\
0\end{matrix}\right),
\]
i.e.,
\[
\left(\begin{matrix}
\hat{T}_x-w &\quad T_{12}\\
0&\quad T_2-w\end{matrix}\right)\left(\begin{matrix}
\xi\\
\eta\end{matrix}\right)=\left(\begin{matrix}
y\\
0\end{matrix}\right).
\]
Since $Ran(\hat{T}_x-w)$ is dense in $\M(x)$, there exists a sequence of vectors $\xi_n$ such that
\[
\lim\limits_{n\rightarrow +\infty}(\hat{T}_x-w)\xi_n+T_{12}\eta=0,
\]
i.e.,
\[
\lim\limits_{n\rightarrow +\infty}(T_x-w)\left(\begin{matrix}
\xi_n\\
\eta\end{matrix}\right)=\left(\begin{matrix}
0\\
0\end{matrix}\right).
\]

Write
\[
\left(\begin{matrix}
\xi_n\\
\eta\end{matrix}\right)=\left(\begin{matrix}
\xi_n-\alpha_n f(w)\\
\eta\end{matrix}\right)+\alpha_n\left(\begin{matrix}
f(w)\\
0\end{matrix}\right)
\]
as an orthogonal decomposition. Note that $(T_x-w)$ is an invertible map from $\left[f(w)\oplus 0\right]^\perp$ onto $Ran(T_x-w) =\H\oplus \H$. Therefore, there exists a $K>0$ such that
\[
\|(T_x-w)\zeta\|\geq K\|\zeta\|,\,\forall \zeta\in\left[\left(\begin{matrix}
f(w)\\
0\end{matrix}\right)\right]^\perp.
\]
So $\lim\limits_{n\rightarrow +\infty}(T_x-w)\left(\begin{matrix}
\xi_n-\alpha_nf(w)\\
\eta\end{matrix}\right)=\left(\begin{matrix}
0\\
0\end{matrix}\right)$ implies that
\[
0=\lim\limits_{n\rightarrow +\infty}\left\|(T_x-w)\left(\begin{matrix}
\xi_n-\alpha_n f(w)\\
\eta\end{matrix}\right)\right\| \geq\limsup_{n\rightarrow +\infty} K\left\|
\left(\begin{matrix}
\xi_n-\alpha_n f(w)\\
\eta\end{matrix}\right)\right\|\geq K\|\eta\|.
\]
Thus $\eta=0$. This is a contradiction. Hence, $\hat{T}_x\in B_1(\Phi_0)$.

Since $\hat{T}_x=T_x|_{\M(x)}$, $\sigma(\hat{T}_x)\subseteq \bar{\mathbb{D}}$. In the following we show $\forall w\in \mathbb{D}\setminus  \sigma(T)^{\land}$, $w\in \rho(\hat{T}_x)$.

By (1) of Lemma  \ref{Theorem 1.1}, $\forall w\in \mathbb{D}\setminus  \sigma(T)^{\land}$,  $$dimker(T_x-w)=1~\mbox{and}~Ran(T_x-w)=\H\oplus \H.$$

Let
$f(w)$ be a nonzero eigenvector of $T_x-w$. Then
\[
f(w)=\left(\begin{matrix}
f_1(w)\\
f_2(w)\end{matrix}\right)
\] with respect to the decomposition $\H\oplus \H=\M(x)\oplus\M(x)^\perp$.

{\bf Claim:}\, $f_2(w)\neq 0$.

 Suppose $f_2(w)=0$. Then $f(w)\in \M(x)$. Note that \[Ran (T_x-w)=\H\oplus\H=\M(x)\oplus \M(x)^\perp\] and $Ran(\hat{T}_x-w)$
is dense in $\M(x)$. So same argument as above shows that $Ran(\hat{T}_x-w)=\M(x)$. Since $T_x-w$ is surjective, $T_2-w$ is also surjective. Assume that $T_2-w$ is not injective. Then there exists an $\eta\in \M(x)^\perp$, $\eta\neq 0$, such that $(T_2-w)\eta=0$. Let $\xi\in \M(x)$ be such that $(\hat{T}_x-w)\xi=-T_{12}\eta$. Then
$$(T_x-w)\left(\begin{matrix}
\xi\\
\eta\end{matrix}\right)=\left(\begin{matrix}
0\\
0\end{matrix}\right).$$ So $\xi\oplus \eta, f(w)\in ker(T_x-w)$ are linearly independent. This contradicts to $dimker(T_x-w)=1$. Thus $T_2-w$ is injective and therefore $T_2-w$ is invertible. So there is a $\delta>0$ such that $T_2-w'$ is invertible for all $w'$ satisfying $|w'-w|<\delta$. Suppose
\[
\left(\begin{matrix}
0\\
0\end{matrix}\right)=(T_x-w')\left(\begin{matrix}
\xi\\
\eta\end{matrix}\right)=
\left(\begin{matrix}
\hat{T}_x-w'&T_{12}\\
0&T_2-w'\end{matrix}\right)\left(\begin{matrix}
\xi\\
\eta\end{matrix}\right)
=\left(\begin{matrix}
\ast\\
(T_2-w')\eta\end{matrix}\right).
\]
Then $(T_2-w')\eta=0$ implies that $\eta=0$. So $ker(T_x-w')\subseteq \M(x)$. Since $w\in \mathbb{D}\setminus  \sigma(T)^{\land}$, $w\in \Sigma$. By (2) of Lemma~\ref{Theorem 1.1},
\[
\H\oplus\H=span\{ker(T_x-w'):\, {|w'-w|<\delta}\}\subset \M(x).
\]
This contradicts to the assumption of the lemma.

Thus
\[
f(w)=\left(\begin{matrix}
f_1(w)\\
f_2(w)\end{matrix}\right)
\] and $f_2(w)\neq 0$. Since $dimker(T_x-w)=1$, $\hat{T}_x-w$ is injective for  $w\in \mathbb{D}\setminus  \sigma(T)^{\land}$. Since $\hat{T}_x\in B_1(\Omega)$, $Ran(\hat{T}_x-w)$ is dense in $\H$ by Lemma~\ref{3.4}. Claim $\hat{T}_x-w$ is surjective. Otherwise, there exists a sequence of unit vectors $\xi_n\in \M$ such that $\lim\limits_{n\rightarrow +\infty}(\hat{T}_x-w)\xi_n=0$.
So
\[
\lim\limits_{n\rightarrow +\infty}(T_x-w)\left(\begin{matrix}
\xi_n\\
0\end{matrix}\right)=\left(\begin{matrix}
0\\
0\end{matrix}\right).
\]
Let
\[
\left(\begin{matrix}
\xi_n\\
0\end{matrix}\right)=\left(\begin{matrix}
\xi_n-\alpha_nf_1(w)\\
-\alpha_nf_2(w)\end{matrix}\right)+\alpha_n\left(\begin{matrix}
f_1(w)\\
f_2(w)\end{matrix}\right)
\]
be an orthogonal decomposition. Then $T_x-w$ is an invertible map from $\left[f_1(w)\oplus f_2(w)\right]^\perp$ onto $Ran(T_x-w)=\H\oplus\H$. Therefore,
\[
\lim\limits_{n\rightarrow +\infty}(T_x-w)\left(\begin{matrix}
\xi_n-\alpha_nf_1(w)\\
-\alpha_nf_2(w)\end{matrix}\right)=\left(\begin{matrix}
0\\
0\end{matrix}\right)
\]
implies that
\[
\lim\limits_{n\rightarrow +\infty}\left\|\left(\begin{matrix}
\xi_n-\alpha_nf_1(w)\\
-\alpha_nf_2(w)\end{matrix}\right)\right\|=0.
\]
Notice that $$\left\|\left(\begin{matrix}
\xi_n-\alpha_nf_1(w)\\
-\alpha_nf_2(w)\end{matrix}\right)\right\|^2=\|\xi_n-\alpha_nf_1(w)\|^2+\|\alpha_nf_2(w)\|^2. $$ So $\lim\limits_{n\rightarrow +\infty} \alpha_n=0$. Thus $\lim\limits_{n\rightarrow +\infty}\left\|\xi_n\oplus 0\right\|=0$, which contradicts to the assumption $\|\xi_n\|=1$ for all $n$. This proves that $\forall w\in \DDD\setminus \sigma(T)^\land$, $w\in \rho(\hat{T}_x)$.
Therefore, $\sigma(\hat{T}_x)\subset  \sigma(T)^{\land}$.

(2). By Lemma~\ref{Theorem 1.1}, $T_x\in B_1(\Sigma)$. Therefore, $Ran(T_x-w)=\H\oplus\H=\M\oplus \M^\perp$, $\forall w\in \Sigma$. It is clear that $Ran(T_2-w)=\M(x)^\perp$, $\forall w\in \Sigma$. Let $f(w)$ be a nonzero eigenvector of $T_x-w$, $w\in\Sigma$. Then
\[
f(w)=\left(\begin{matrix}
f_1(w)\\
f_2(w)\end{matrix}\right)
\] and $f_2(w)\neq 0$. It is clear that $(T_2-w)f_2(w)=0$. Assume that $(T_2-w)\eta=0$. Then
\[
(T_x-w)\begin{pmatrix}
-(\hat{T}_x-w)^{-1} T_{12}(\eta)\\
\eta
\end{pmatrix}=0.
\]
Since $dimker(T_x-w)=1$, $f_1(w)\oplus f_2(w)$ and $
(-(\hat{T}_x-w)^{-1} T_{12}(\eta))\oplus
\eta$ are linearly dependent. Therefore $f_2(w)$ and $\eta$ are linearly dependent. This implies that $dimker(T_2-w)=1$.
 Since
  $span\{f(w):\,w\in \Sigma\}=\H\oplus\H$,
 $span\{f_2(w):\, w\in \Sigma\}=\M(x)^\perp$. Hence, $T_2\in B_1(\Sigma)$.

 By Lemma~\ref{Theorem 1.1}, $RanT_x=\H\oplus\H=\M(x)\oplus \M(x)^\perp$, so $RanT_2=\M(x)^\perp$. Also we have $dimker T_x=1$. This implies that $ker T_x=ker\hat{T}_x$. Since $Ran\hat{T}_x=\M(x)$, $T_2$ is injective. This proves that $T_2$ is invertible.
\end{proof}

\begin{lem}\label{Lemma 2.1}
Let $T\in\mathcal{B}(\H)$ be an invertible bounded linear operator. For $x\in \H$, let $\N_{-k}(x)=span\{T^{-n}x:\, n\geq k\}$. Then $\cap_{k=1}^\infty \N_{-k}(x)\in Lat(T)$.
\end{lem}
\begin{proof}
Choose $y\in \cap_{k=1}^\infty \N_{-k}(x)$, we  prove $Ty\in \cap_{k=1}^\infty \N_{-k}(x)$. Since $y\in \N_{-2}(x)$, $\forall\epsilon>0$, there exists $\sum\limits_{i=2}^{m_2}\alpha_iT^{-i}x$, $\alpha_i\in \mathbb{C}$, such that
\[
\left\|\sum\limits_{i=2}^{m_2}\alpha_iT^{-i}x-y\right\|<\epsilon.
\]
Thus
\[
\left\|\sum\limits_{i=2}^{m_2}\alpha_i T^{-i+1}x-Ty\right\|=\left\|T\left(\sum\limits_{i=2}^{m_2}\alpha_iT^{-i}x-y\right)\right\|\leq \|T\|\left\|\sum\limits_{i=2}^{m_2}\alpha_iT^{-i}x-y\right\|<\|T\|\epsilon.
\]
It follows that $Ty\in \N_{-1}(x)$. Similarly, we can show that $Ty\in \N_{-k}(x)$ for all $k\geq 1$ and therefore, $Ty\in \cap_{k=1}^\infty \N_{-k}(x)$.
\end{proof}

\begin{lem}\label{Lemma 5.1}
Suppose $T\in\B(\H)$ is an invertible transitive  operator. Let $\N_{-k}(x)=span\{T^{n}x: n\leq -k\}$ for $k\geq 1$. If $\N(x)\triangleq\N_{-1}(x)\neq \H$, then $\{\N_{-k}(x)\}$ is a strictly decreasing sequence and $\cap_{k=1}^\infty \N_{-k}(x)=\{0\}$.
\end{lem}
\begin{proof}
By Lemma~\ref{Lemma 2.1}, $\cap_{k=1}^\infty \N_{-k}(x)\in Lat T$. Since $\N(x)=\N_{-1}(x)\neq \H$, $\cap_{k=1}^\infty \N_{-k}(x)=\{0\}$. Suppose $\N_{-k}(x)=\N_{-(k+1)}(x)$ for some $k\geq 1$. Then $T^{-k}x\in \N_{-(k+1)}(x)$. Hence, $T^{-(k+1)}x\in T^{-1}\N_{-(k+1)}(x)=\N_{-(k+2)}(x)$. This implies that $\N_{-(k+2)}(x)=\N_{-(k+1)}(x)=\N_{-k}(x)$. By induction, we have $\N_{-n}(x)=\N_{-k}(x)$ for all $n\geq k$.  This contradicts to $\cap_{k=1}^\infty \N_{-k}(x)=\{0\}$.
\end{proof}

\begin{lem}\label{Theorem 5.5}
Suppose an invertible operator $T\in\B(\H)$ is  transitive, $r(T)<1$,   and \[\N(x)\triangleq\N_{-1}(x)=span\{T^{n} x:\,n\leq -1\}\neq \H.\] Then the map $B_x: (-T^{-(n+1)}x)\oplus e_n\rightarrow -T^{-(n+1)}x$, $n\geq 0$, extends to a bounded linear isomorphism from $\M(x)=span\{(-T^{-(n+1)}x)\oplus e_n:\,n\geq 0\}$ onto $\N(x)$.
\end{lem}
\begin{proof}
Recall that $a_n=P_{\M(x)}(0\oplus e_n)$ (see Lemma~\ref{start}).
For $n\geq 0$, define $f_n(x)=\langle x,a_n\rangle$ for $x\in \H\oplus \H$. By Lemma~\ref{Theorem 1.5}, $\sigma(\hat{T}_x^* )\subset \sigma(T)^{\land}$, so the spectral radius of $\hat{T}_x^*$ is strictly less than a positive number $r<1$.   By Lemma \ref{start},  $\hat{T}_x^*a_n=a_{n+1}$ and $a_n=\left(\hat{T}_x^*\right)^n a_0$ for $n\geq 0$.
Thus $\|f_n\|=\|a_n\|\leq\left(r+\epsilon\right)^n\|a_0\|$ for sufficient large $N$ such that $\sum\limits_{n=N+1}^\infty \|a_n\|=\theta<1$. For each $y\in \H\oplus\H$, put
\[
A(y)=y+\sum\limits_{n=N+1}^\infty f_n(y)((-T^{-(n+1)}x)\oplus 0-(-T^{-(n+1)}x)\oplus e_n).
\]
Note that $$f_n((-T^{-(k+1)}x)\oplus e_k)=\langle (-T^{-(k+1)}x)\oplus e_k, a_n\rangle=\langle (-T^{-(k+1)}x)\oplus e_k, P_{\M(x)}(0\oplus e_n)\rangle$$
$$=\langle P_{\M(x)}(-T^{-(k+1)}x)\oplus e_k, 0\oplus e_n\rangle=\langle (-T^{-(k+1)}x)\oplus e_k, 0\oplus e_n\rangle=\delta_{k,n}. $$
Then for $k\geq N+1$, we have
\[
A((-T^{-(k+1)}x)\oplus e_k)\]\[=(-T^{-(k+1)}x)\oplus e_k+\sum\limits_{n=1}^\infty f_n((-T^{-(k+1)}x)\oplus e_k)((-T^{-(n+1)}x)\oplus 0-(-T^{-(n+1)}x)\oplus e_n)\]\[=(-T^{-(k+1)}x)\oplus 0.
\]
Note that
\[
\|A-1\|\leq \sum\limits_{n=N+1}^\infty\|f_n\|\|(-T^{-(n+1)}x)\oplus 0-(-T^{-(n+1)}x)\oplus e_n\|=\sum\limits_{n=N+1}^\infty \|a_n\|=\theta<1.
\]
So $A$ is an invertible bounded linear operator on $\H\oplus\H$ and $A$ maps $$span\{(-T^{-(n+1)}x)\oplus e_n:n\geq N+1\}$$ onto $span\{-T^{-(n+1)}x:n\geq N+1\}$.

{\bf Claim}\,\, $span\{T^{-1}x,\cdots,T^{-(N+1)}x\}\cap span\{T^{-(N+2)}x,T^{-(N+3)}x,\cdots\}=\{0\}$. Suppose \[w\in span \{T^{-1}x, \cdots, T^{-(N+1)}x\}\] and \[v\in span\{T^{-(N+2)}x,T^{-(N+3)}x,\cdots\}\] satisfy $w=v$. If $w\neq 0$, then $w=\sum\limits_{i=1}^{N+1}\alpha_i T^{-i}x=z$. We may assume that $\alpha_{j}=-1$ and $\alpha_{i}=0$ for $1\leq i<j$. Then
\[
T^{-j}x=\sum\limits_{i=j+1}^{N+1} \alpha_i T^{-i}x+z\in \N_{-(j+1)}(x).
\]
This implies that $\N_{-j}(x)=\N_{-(j+1)}(x)$, which contradicts to Lemma~\ref{Lemma 5.1}. Let $x'\in \N(x)$. Then there exist a sequence of vectors $\{x_n\}\in \N(x)$ such that each $x_n$ can be written as a finite linear combinations of $\{T^{-n}x\}_{n=1}^\infty$ and $\lim\limits_{n\rightarrow\infty}\|x_n-x'\|=0$. In particular, there is $\{y_n\}\in span\{T^{-1}x,T^{-2}x,\cdots,T^{-(N+1)}x\}$ and $z_n\in span\{T^{-(N+2)}x,T^{-(N+3)}x,\cdots\}$ such that $x_n=y_n+z_n$. Let $\pi$ be the quotient map from $\N$ onto $\L/span\{T^{-(N+2)}x,T^{-(N+3)}x,\cdots\}$. Then $\pi(y_n)=\pi(x_n)$ is a Cauchy sequence in  $\L/span\{T^{-(N+2)}x,T^{-(N+3)}x,\cdots\}$. Clearly, $\pi$ is a surjective map from $span\{T^{-1}x,T^{-2}x,\cdots,T^{-(N+1)}x\}$ onto $span\{\pi(T^{-1}x),\pi(T^{-2}x),\cdots,\pi(T^{-(N+1)}x)\}$. Suppose $\pi(w_1T^{-1}x+w_2T^{-2}x+\cdots+w_{N+1}T^{-(N+1)}x)=0$. Then there exists a \[z\in span\{T^{-(N+2)}x,T^{-(N+3)}x,\cdots\}\] such that $w_1T^{-1}x+w_2T^{-2}x+\cdots+w_{N+1}T^{-(N+1)}x=z$. By the above argument, $w_1T^{-1}x+w_2T^{-2}x+\cdots+w_{N+1}T^{-(N+1)}x=z=0$. So $\pi$ is an injective map from  \[span\{T^{-1}x,T^{-2}x,\cdots,T^{-(N+1)}x\}\] onto \[span\{\pi(T^{-1}x),\pi(T^{-2}x),\cdots,\pi(T^{-(N+1)}x)\}.\] Since $span\{T^{-1}x,T^{-2}x,\cdots,T^{-(N+1)}x\}$ is finite dimensional, $\{y_n\}$ is a Cauchy sequence. Thus $\{z_n\}$ is also a Cauchy sequence. Let $y=\lim\limits_{n\rightarrow\infty} y_n$ and $z=\lim\limits_{n\rightarrow\infty}z_n$. Then \[y\in span\{T^{-1}x,T^{-2}x,\cdots,T^{-(N+1)}x\}\] and \[z\in span\{T^{-(N+2)}x,T^{-(N+3)}x,\cdots\}\] such that $x'=y+z$. This implies that
\[
\N(x)=span\{T^{-(n+1)}x:0\leq n\leq N\}\overset{\cdot}{+}span\{T^{-(n+1)}x:n\geq N+1\}.
\]
Similarly,
\[
\M(x)=span\{ (-T^{-(n+1)}x)\oplus e_n:0\leq n\leq N\}\overset{\cdot}{+}span\{ (-T^{-(n+1)}x)\oplus e_n:n\geq N+1\}.
\]

Hence $B_x: (-T^{-(n+1)}x)\oplus e_n\rightarrow -T^{-(n+1)}x$ extends to a bounded linear isomorphism from $\M(x)$ onto $\N(x)$.
\end{proof}

\begin{lem}\label{Theorem 4.1}
Suppose an invertible operator $T\in\B(\H)$, $r(T)<1$, $x$ is a cyclic vector of $T$ and $\M(x)\neq \H\oplus \H$. If $b_0\neq 0$, then $P_{\M(x)^\perp}$ is an invertible bounded linear  operator from $0\oplus \H$ onto $\L(x)=span\{b_0,b_1,\cdots\}$, where $b_n=P_{\M(x)^\perp}(0\oplus e_n)$ {\rm(see in Lemma \ref{start})}.
\end{lem}
\begin{proof}
 By Lemma \ref{start},  $\hat{T}_x^*a_n=a_{n+1}$ for $n\geq 0$. Then we have $a_n=\left(\hat{T}_x^*\right)^n a_0$ for $n\geq 0$ and $$\|a_n\|\leq\left\|\left(\hat{T}_x^*\right)^n \right\|  \|a_0\|.$$ By Lemma~\ref{Theorem 1.5}, $\sigma(\hat{T}_x^* )\subset \sigma(T)^{\land}$, so the spectral radius of $\hat{T}_x^*$ is  a positive number $r<1$.  For arbitrary $\epsilon>0$, there exists an $N_1$ such that $\|a_n\|\leq\left|\left(r+\epsilon\right)^n \right| \|a_0\|$ for $n\geq N_1$. This implies that there exists an $N$ sufficiently large such that \[\sum\limits_{n=N+1}^\infty \|(0\oplus e_n)-b_n\|=\sum\limits_{n=N+1}^\infty \|a_n\|<\frac{1}{2}.\] By Theorem 1.3.9 of \cite{AK}, there exists an invertible bounded linear operator $A$ from \[span\{e_{N+1},e_{N+2},\cdots\}\] onto \[span\{b_{N+1},b_{N+2},\cdots\}\] such that $A(e_k)=b_k$ for $k=N+1,N+2,\cdots$. So $A$ is the restriction of $P_{\M(x)^\perp}$ onto $span\{e_{N+1},e_{N+2},\cdots\}$ and $ \{b_{N+1},b_{N+2},\cdots\}$ is a basic sequence.

Suppose $P_{\M(x)^\perp}(0\oplus z)=0$. Write $z=\sum\limits_{n=0}^\infty w_n (0\oplus e_n)$, where $\sum\limits_{n=0}^\infty |w_n|^2<\infty$. Then
\[
0=P_{\M(x)^\perp}(0\oplus z)=P_{\M(x)^\perp}\left(\sum\limits_{n=0}^\infty w_n (0\oplus e_n) \right)=\sum\limits_{n=0}^\infty w_n b_n.
\]
If $z\neq 0$, then $w_k\neq 0$ for some $k$ and $w_j=0$ for all $j<k$. We may assume that $w_k=-1$. Then $b_k=\sum\limits_{n=k+1}^\infty w_n b_n$. Recall that $S_2 b_n=b_{n+1}$ for all $n\geq 0$ (see Lemma~\ref{start}). Thus
\[
b_{N+1}=S_2^{N+1-k}b_k=\sum\limits_{n=k+1}^\infty w_n S_2^{N+1-k}b_n=\sum\limits_{n=k+1}^\infty w_n b_{n+N+1-k}.
\] This contradicts to the fact that $\{b_{N+1},b_{N+2},\cdots\}$ is a basic sequence.
Hence, $z=0$. So
 we have $P_{\M(x)^\perp}$ is an injective bounded linear operator from $0\oplus \H$ into $\L=span\{b_0,b_1,\cdots\}$. We need only to show the map is surjective.

 Suppose $w\in span\{b_0,b_1,\cdots,b_N\}$ and $v\in span\{b_{N+1},b_{N+2},\cdots\}$ satisfy $w=v$. Write $w=\sum\limits_{i=0}^N\alpha_i b_i$ and $v=\sum\limits_{i=N+1}^\infty \alpha_i b_i$, where $\sum\limits_{i=N+1}^\infty |\alpha_i|^2<\infty$. Let
 \[z=\sum\limits_{i=0}^N\alpha_i e_i-\sum\limits_{i=N+1}^\infty \alpha_i e_i.\]
 Then $P_{\M^\perp}(0\oplus z)=w-v=0$. This implies that $z=0$ and therefore, $w=v=0$. So $$span\{b_0,b_1,\cdots,b_N\}\cap span\{b_{N+1},b_{N+2},\cdots\}=\{0\}.$$ Let $x'\in\mathcal{L}$. Then there exists a sequence of vectors $\{x_n\}\in\mathcal{L}$ such that each $x_n$ can be written as a finite linear combinations of $\{b_n\}_{n=1}^\infty$ and $\lim\limits_{n\rightarrow\infty}\|x_n-x'\|=0$. In particular, there is $y_n\in span\{b_0,b_1,\cdots,b_N\}$ and $z_n\in span\{b_{N+1},b_{N+2},\cdots\}$ such that $x_n=y_n+z_n.$

 Let $\pi$ be the quotient map from $\L(x)$ onto $\L(x)/span\{b_{N+1},b_{N+2},\cdots\}$. Then $\pi(y_n)=\pi(x_n)$ is a Cauchy sequence in  $\L(x)/span\{b_{N+1},b_{N+2},\cdots\}$. Clearly, $\pi$ is a surjective map from $span\{b_0,b_1,\cdots,b_N\}$ onto $span\{\pi(b_0),\pi(b_1),\cdots,\pi(b_N)\}$.

  Suppose $\pi(w_0b_0+w_1b_1+\cdots+w_Nb_N)=0$. Then there exists a $v\in span\{b_{N+1},b_{N+2},\cdots\}$ such that $$w_0b_0+w_1b_1+\cdots+w_Nb_N=v.$$ Therefore, $$w_0b_0+w_1b_1+\cdots+w_Nb_N=v=0.$$ So $\pi$ is an injective map from  $span\{b_0,b_1,\cdots,b_N\}$ onto $span\{\pi(b_0),\pi(b_1),\cdots,\pi(b_N)\}$.

  Since $span\{b_0,b_1,\cdots,b_N\}$ is finite dimensional, $\{y_n\}$ is a Cauchy sequence. Thus $\{z_n\}$ is also a Cauchy sequence. Let $y=\lim\limits_{n\rightarrow\infty} y_n$ and $z=\lim\limits_{n\rightarrow\infty}z_n$. Then $y\in span\{b_0,b_1,\cdots,b_N\}$ and $z\in span\{b_{N+1},b_{N+2},\cdots\}$ such that $x'=y+z$. This implies that $$\L(x)=span\{b_0,b_1,\cdots,b_N\}\overset{.}{+}span\{b_{N+1},b_{N+2},\cdots\}$$ and $P_{\M(x)^\perp}$ is onto.
\end{proof}

\begin{lem}\label{Theorem 2.5}
Let $\L(x)=span\{b_n,\,n\geq 0\}$, $\N(x)=span\{T^{-(n+1)}x,\,n\geq 0\}$. Then
\[
(\M(x)\oplus\mathcal{L}(x))\oplus (\N(x)^\perp \oplus 0)=\H\oplus \H,
\]
where $b_n=P_{\M(x)^\perp}(0\oplus e_n)$ {\rm (see in Lemma \ref{start}.)}
\end{lem}
\begin{proof}
We need only to prove $\M(x)\oplus\mathcal{L}(x)=\N(x)\oplus \H$.  First, we show that \[\M(x)\oplus\mathcal{L}(x)\subseteq \N(x)\oplus \H.\] Clearly, $(-T^{-(n+1)}x)\oplus e_n\in \N(x)\oplus \H$ for all $n\geq 0$. Thus $\M(x)\subset \N(x)\oplus \H$. Note that $b_n=(0\oplus e_n)-a_n\in \N(x)\oplus \H$. Thus $\L(x)\subset \N(x)\oplus \H$ and $\M(x)\oplus\mathcal{L}(x)\subseteq \N(x)\oplus\H$. Second, we show that $\N(x)\oplus\H\subseteq \M(x)\oplus\mathcal{L}(x)$. Note that for $n\geq 0$, \[0\oplus e_n=a_n+b_n\in \M(x)\oplus\mathcal{L}(x).\] Also for $n\geq 0$, $(-T^{-(n+1)}x)\oplus 0=(-T^{-(n+1)}x)\oplus e_n-(0\oplus e_n)\in \M(x)\oplus\mathcal{L}(x)$. Thus \[\N(x)\oplus\H\subseteq \M(x)\oplus\mathcal{L}(x).\]
\end{proof}

\begin{cor}\label{orthogonal decomposition}
$\L(x)\subset \M(x)^\perp$ and $\M(x)^\perp\ominus \L(x)=\N(x)^\perp\oplus 0$.
\end{cor}

By equation (5.2), Lemma~\ref{Theorem 2.5} and Corollary~\ref{orthogonal decomposition}, we have the following
\[
S_x=\left(\begin{matrix}
\hat{S}_x&\quad S_{1,2}\\
0&\quad S_2
\end{matrix}\right)\begin{matrix}
\M(x)\\
\M(x)^\perp
\end{matrix}=\left(\begin{matrix}
\hat{S}_x&\quad \bar{S}_{12}&\quad\bar{S}_{13}\\
0&\quad S_2|_{\L(x)}&\quad G_{12}\\
0&\quad 0&\quad P_{\N(x)^\perp\oplus 0}S_2 P_{\N(x)^\perp\oplus 0}\end{matrix}
\right)\begin{matrix} \M(x)\\\L(x)\\\N(x)^\perp\oplus 0\end{matrix}\]
\[
=\left(\begin{matrix}
T^{-1}&\quad 0\\
0&\quad S_1
\end{matrix}\right)\begin{matrix}
\H\\
\H
\end{matrix}=\left(\begin{matrix}
\bar{T}_1(x)&\quad \bar{T}_{12}(x)&\quad 0\\
0&\quad \bar{T}_2(x)&\quad 0\\
0&\quad 0&\quad S_1
\end{matrix}
\right)\begin{matrix}
\N(x)\oplus 0\\\N(x)^\perp\oplus 0\\0\oplus\H
\end{matrix}.
\]
In particular,
\begin{equation}
T^{-1}=\begin{pmatrix}
\bar{T}_1(x)&\quad\bar{T}_{12}(x)\\
0&\quad \bar{T}_2(x)
\end{pmatrix}
\begin{matrix}
\N(x)\\
\N(x)^\perp
\end{matrix},
\end{equation}
where $\bar{T}_2(x)=P_{\N(x)^\perp\oplus 0}S_2 P_{\N(x)^\perp\oplus 0}$.

Recall that \[T_x=\left(\begin{matrix}
T&\quad x\otimes e_0\\
0&\quad S_1^*
\end{matrix}\right)\begin{matrix}
\H\\
\H
\end{matrix}=\left(\begin{matrix}
\hat{T}_x&\quad T_{1,2}\\
0&\quad T_2
\end{matrix}\right)\begin{matrix}
\M(x)\\
\M(x)^\perp
\end{matrix}\] and  \[S_x=\left(\begin{matrix}
T^{-1}&\quad 0\\
0&\quad S_1
\end{matrix}\right)\begin{matrix}
\H\\
\H
\end{matrix}=\left(\begin{matrix}
\hat{S}_x&\quad S_{1,2}\\
0&\quad S_2
\end{matrix}\right)\begin{matrix}
\M(x)\\
\M(x)^\perp
\end{matrix}.\]

\begin{lem}\label{Theorem 3.1}
$\hat{S}_x^*\in B_1(s\DDD)$, where $s=r(T)^{-1}$.
\end{lem}
\begin{proof}
Since $T_xS_x=I$, $S^*_xT^*_x=I$. We have
\[
\left(\begin{matrix}
\hat{S}^*_x&\quad 0\\
S^*_{1,2}&\quad S^*_2
\end{matrix}\right)\left(\begin{matrix}
\hat{T}^*_x&\quad 0\\
T^*_{1,2}&\quad T^*_2
\end{matrix}\right)=I.
\]
This implies that $\hat{S}^*_x\hat{T}^*_x=I$. By Lemma~\ref{Theorem 1.5} and Theorem 16.12 of~\cite{MU}, \[ind\hat{S}^*_x=-ind\hat{T}^*_x=ind\hat{T}_x=1.\] For $\lambda\in \CCC$, we have
\[
\left(\hat{S}^*_x-\lambda\right)\hat{T}^*_x=I-\lambda\hat{T}^*_x=\lambda\left(\lambda^{-1}-\hat{T}^*_x\right).
\]
By Lemma~\ref{Theorem 1.5}, $\sigma\left(\hat{T}^*_x\right)\subseteq \sigma(T)^\land\subseteq r(T)\bar{\DDD}$. Hence, for $|\lambda|<s$, $|\lambda^{-1}|>r(T)$ and
$I-\lambda\hat{T}^*_x$ is invertible. This implies that $\hat{S}^*_x-\lambda$ is surjective for $|\lambda|<s$. Therefore, $dimker(\hat{S}_x-\lambda)=0$ for $|\lambda|<s$. By the continuity of index, $ind(\hat{S}^*_x-\lambda)=1$ for all $|\lambda|<s$. So $dimker(\hat{S}^*_x-\lambda)=1$ for all $|\lambda|<s$.

By Lemma~\ref{start}, $\hat{S}^*_x a_0=0$.
Let
\[
e(w)=\sum_{n=0}^\infty w^n \left(\hat{T}^*_x\right)^n(a_0).
\]
Since  $r(\hat{T}^*_x)\leq r(T)$, $e(w)$ is well-defined for $|w|<s$. We have
\[
\hat{S}^*_x(e(w))=\hat{S}^*_x\left(\sum_{n=0}^\infty w^n \left(\hat{T}^*_x\right)^n(a_0)\right)=w\left(\sum_{n=0}^\infty w^n \left(\hat{T}^*_x\right)^n(a_0)\right)=we(w).
\]
Suppose
\[
\sum_{n=0}^\infty w_0^n \left(\hat{T}^*_x\right)^n(a_0)=0
\]
for some $|w_0|<s$. Then
\[
0=\hat{T}^*_x\left(\sum_{n=0}^\infty w_0^n \left(\hat{T}^*_x\right)^n(a_0)\right)=\sum_{n=0}^\infty w_0^n \left(\hat{T}^*_x\right)^{n+1}(a_0).
\]
Hence
\[
0=\sum_{n=0}^\infty w_0^{n+1} \left(\hat{T}^*_x\right)^{n+1}(a_0)=\sum_{n=1}^\infty w_0^n \left(\hat{T}^*_x\right)^{n}(a_0).
\]
Therefore,
\[
a_0=0.
\]
This is a contradiction.
By Lemma~\ref{Lemma 1.4},
\[
span\{e_n(w):\,{|w|<s}\}=span\left\{ker \left(\hat{S}^*_x\right)^n:\,n\geq 1\right\}=span\{a_0,a_1,\cdots,\}=\M,
\]
we conclude that $\hat{S}^*_x\in B_1(s\DDD)$.
\end{proof}

\begin{lem}\label{Theorem 3.2}
Let $\Sigma$ be the connected component of $\DDD\setminus \sigma(T)$ which contains $\{w\in \DDD: r(T)<|w|<1\}$ and let $\Sigma^{-1}=\{w^{-1}:\,w\in \Sigma\}$.
Then $S_2\in B_1(\Sigma^{-1})$.
\end{lem}
\begin{proof}
Since $T_xS_x=I$, we have
\[
\left(
\begin{matrix}
\hat{T}_x&\quad T_{12}\\
0&\quad T_2
\end{matrix}
\right)\left(
\begin{matrix}
\hat{S}_x&\quad S_{12}\\
0&\quad S_2
\end{matrix}
\right)=I.
\]
This implies that $T_2S_2=I$. By Lemma~\ref{Theorem 1.5}, for $w\in \Sigma^{-1}$, we have $w^{-1}\in \Sigma$ and
\[
T_2(S_2-w)=I-w T_2=w(w^{-1}-T_2)
\]
is surjective. By Lemma~\ref{Theorem 1.5}, $T_2$ is invertible, so $S_2-w$ is surjective  for $w\in \Sigma^{-1}$. Also note that $\forall w\in \Sigma^{-1}$, $w^{-1}\in \Sigma$,
\[
(S_2^*-\bar{w})T_2^*=I-\bar{w}T_2^*=\bar{w}(w^{-1}-T_2)^*.
\]
So $Ran(S_2^*-\bar{w})=Ran(w^{-1}-T_2)^*$. This implies that
\[
ker(S_2-w)=(Ran(S_2^*-\bar{w}))^\perp=(Ran(w^{-1}-T_2)^*)^\perp=ker(w^{-1}-T_2)
\]
is dimensional one and by Lemma~\ref{Theorem 1.5},
\[
span\{ker(S_2-w):\,{w\in \Sigma^{-1}}\}=span\{ker(w^{-1}-T_2):\,{w^{-1}\in \Sigma}\}=\M(x)^\perp.
\]
So $S_2\in B_1(\Sigma^{-1})$.
\end{proof}

\begin{lem}\label{Theorem 6.2}
Let $S_2=\left(\begin{matrix}
S_2|_{\L(x)}&G_{12}\\
0& P_{\N(x)^\perp\oplus 0}S_2 P_{\N(x)^\perp\oplus 0}
\end{matrix}\right)\begin{matrix} \L(x)\\ \N(x)^\perp\oplus 0\end{matrix}
.$ Then $S_2|_{\L(x)}$ is similar to the unilateral shift operator and
 $P_{\N(x)^\perp\oplus 0}S_2 P_{\N(x)^\perp\oplus 0}\in B_1(\Sigma^{-1})$.
\end{lem}
\begin{proof}
By Lemma~\ref{Theorem 4.1},
\[
(P_{\M(x)^\perp})^{-1}S_2|_{\L(x)} P_{\M(x)^\perp}(0\oplus e_n)=(P_{\M(x)^\perp})^{-1}S_2|_{\L(x)} b_n=(P_{\M(x)^\perp})^{-1}b_{n+1}=0\oplus e_{n+1},
\]
for all $n\geq 0$. So $S_2|_{\L(x)}$ is similar to the unilateral shift operator.

 By Lemma~\ref{Theorem 3.2}, $S_2-w$ is surjective for $w\in \Sigma^{-1}$. Therefore, $\left(P_{\N(x)^\perp\oplus 0}S_2 P_{\N(x)^\perp\oplus 0}-w\right)$ is also surjective.
For $w\in \Sigma^{-1}$, there exists a nonzero vector $f(w)=f_1(w)\oplus f_2(w)$ such that
\[
\left(\begin{matrix}
(S_2-w)|_{\L(x)}&G_{12}\\
0& P_{\N(x)^\perp\oplus 0}S_2 P_{\N(x)^\perp\oplus 0}
\end{matrix}
\right)\left(\begin{matrix}
f_1(w)\\
f_2(w)\end{matrix}\right)=\left(\begin{matrix}
0\\
0\end{matrix}\right).
\]
Suppose $f_2(w)=0$. Then $f_1(w)\neq 0$ and
 $(S_2|_{\L(x)}-w)f_1(w)=0$. So $S_2|_{\L(x)}-w$ is not invertible. This is a contradiction. Thus $f_2(w)\neq 0$ and $ P_{\M(x)^\perp\ominus \L(x)}(S_2-w) P_{\M(x)^\perp\ominus \L(x)}(f_2(w))=0$. Suppose $\eta_1,\eta_2\in ker\left(P_{\N(x)^\perp\oplus 0}S_2 P_{\N(x)^\perp\oplus 0}-w\right)$. Then there exist $\xi_1,\xi_2\in \L(x)$ such that $$(S_2-w)\left(\begin{matrix} \xi_i\\\eta_i\end{matrix}\right)=\left(\begin{matrix}0\\0\end{matrix}\right).$$ So $\xi_1\oplus \eta_1,\xi_2\oplus\eta_2\in ker(S_2-w)$. Since $dimker(S_2-w)=1$, $\xi_1\oplus \eta_1,\xi_2\oplus\eta_2$ are linearly dependent. Therefore, $\eta_1,\eta_2$ are linearly dependent. So \[dimker\left(P_{\N(x)^\perp\oplus 0}S_2 P_{\N(x)^\perp\oplus 0}-w\right)=1.\] Since $span\left\{f_1(w)\oplus f_2(w):\,{w\in \Omega_1^{-1}}\right\}=\M(x)^\perp$, $span\{f_2(w):\,{w\in \Omega_1^{-1}}\}=\M(x)^\perp\ominus \L(x)$. This implies that  $$P_{\N(x)^\perp\oplus 0}S_2 P_{\N(x)^\perp\oplus 0}\in B_1(\Sigma^{-1}).$$
\end{proof}

\begin{lem}\label{Theorem 6.3}
Write
\[
T^{-1}=\left(
\begin{matrix}
\bar{T}_1(x)&\quad \bar{T}_{12}(x)\\
0&\quad \bar{T}_2(x)
\end{matrix}
\right)\begin{matrix}
\N(x)\\
\N(x)^\perp
\end{matrix}.
\]
Then $\bar{T}_2(x)=P_{\M(x)^\perp\ominus \L(x)}S_2 P_{\M(x)^\perp\ominus \L(x)}\in B_1(\Sigma^{-1})$ and $\bar{T}_1^*(x)\in B_1(\Omega_0^*)$, where $\Omega_0$ is the connected component of $\CCC\setminus \sigma(T^{-1})$ containing 0.
\end{lem}
\begin{proof}
By Lemma~\ref{Theorem 5.5},
\[
B_x\hat{S}_x B_x^{-1}(T^{-n}x)=B_x\hat{S}_x(T^{-n}x\oplus (-e_{n-1}))=B_x(T^{-(n+1)}x\oplus (-e_{n}))=T^{-(n+1)}x.
\]
This implies that $\bar{T}_1(x)=B_x\hat{S}_x B_x^{-1}$. By Lemma~\ref{Theorem 3.1}, we have $\bar{T}_1^*(x)\in B_1(s\DDD)$, where $s=r(T)^{-1}$. For $\lambda\in \Omega_0$, $\lambda\in \rho(T^{-1})$. Hence, $\lambda\in \rho_{s-F}(\bar{T}_1(x))$. Since $ind \bar{T}_1(x)=0$. The continuity of index implies that $ind \left(\bar{T}_1(x)-\lambda\right)=-1$ for all $\lambda\in \Omega_0$. Since $\bar{T}_1^*\in B_1(s\DDD)$, Lemma~\ref{3.4} implies that $ker\left(\bar{T}_1(x)-\lambda\right)=\{0\}$ and $dim ker\left(\bar{T}_1^*(x)-\bar{\lambda}\right)=1$ for all $\lambda\in \Omega_0$. Thus  $\bar{T}_1^*(x)\in B_1(\Omega_0^*)$. By Lemma~\ref{Theorem 6.2},  $\bar{T}_2(x)=P_{\M(x)^\perp\ominus \L(x)}S_2 P_{\M(x)^\perp\ominus \L(x)}=P_{\N(x)^\perp\oplus 0}S_2 P_{\N(x)^\perp\oplus 0}\in B_1(\Sigma^{-1})$. This proves the lemma.
\end{proof}

Recall that $\Phi_0$ is the connected component of $\CCC\setminus \sigma(T)$ containing 0 and $\Omega_0$ is the connected component of $\CCC\setminus \sigma(T^{-1})$ containing 0. Note that
\[
\Phi_0=\left\{\frac{1}{\lambda}:\, \lambda\in\CCC\setminus \sigma(T^{-1})^\land\cup\{\infty\}\right\}.
\]
By Lemma~\ref{Theorem 5.5}, the map $B_x: (-T^{-(n+1)}x)\oplus e_n\rightarrow -T^{-(n+1)}x$, $n\geq 0$, extends to a bounded linear isomorphism from $\M(x)=span\{(-T^{-(n+1)}x)\oplus e_n:\,n\geq 0\}$ onto $\N(x)=span\{T^{-(n+1)}x:\, n\geq 0\}$.

\begin{prop}\label{P:main}
Suppose $T$ is transitive, $r(T)<1$, $x$ is a unit noncyclic vector of $T^{-1}$. Then
\[
T=\begin{pmatrix}
H_1(x)&\quad H_{12}(x)\\
H_{21}(x)&\quad H_2(x)
\end{pmatrix}\begin{matrix}
\N(x)\\
\N^\perp(x)
\end{matrix},\quad T^{-1}=\begin{pmatrix}
\bar{T}_1(x)&\quad \bar{T}_{12}(x)\\
0&\quad \bar{T}_2(x)
\end{pmatrix}\begin{matrix}
\N(x)\\
\N^\perp(x)
\end{matrix}
\]
satisfy
\begin{enumerate}
\item $\bar{T}_1(x)=B_x\hat{S}_xB_x^{-1}$, $\bar{T}_1^*(x)\in B_1(\Omega_0^*)$, and $T^{-1}x\in\C(\bar{T}_1(x))$;
\item $\bar{T}_2(x)\in B_1(\Omega_0)$, $\bar{T}_2(x)H_2(x)=I_{\N(x)^\perp}$, and $\C(H_2(x))\neq \emptyset$;
\item $\bar{H}_1(x)\triangleq B_x\hat{T}_xB_x^{-1}\in B_1(\Phi_0)$, $\bar{H}_1(x)\bar{T}_1(x)=I_{\N(x)}$, and $\C(\bar{H}_1^*(x))\neq \emptyset$;
\item $H_{21}(x)$ is a rank one operator.
\end{enumerate}
\end{prop}
\begin{proof}
(1) $\bar{T}_1^*(x)\in B_1(\Omega_0^*)$ and $\bar{T}_1(x)=B_x\hat{S}_xB_x^{-1}$ follow from  Lemma~\ref{Theorem 6.3} and its proof. Note that
\[
\N(x)=span\{T^{-k}x:\,k\geq 1\}=span\{\bar{T}_1(x)^k T^{-1}x: k\geq 0\}.
\]
Therefore, $T^{-1}x\in\C(\bar{T}_1(x))$.

(2)By Lemma~\ref{Theorem 6.2} and Lemma~\ref{Theorem 6.3}, $\overline{T}_2(x)\in B_1(\Phi)$, where
\[
\Phi=\left\{w,\quad 1<|w|<\frac{1}{r(T)}\right\}.
\]
Note that $\Phi\subseteq \Omega_0$. Since $\Omega_0\subset \rho(T^{-1})$ and $\overline{T}_1^*(x)\in B_1(\Omega_0)$, $\Omega_0\subseteq \rho_F(\overline{T}_2(x))$. By the continuity of index, we deduce that $\overline{T}_2(x)\in B_1(\Omega_0)$. Since $\bar{T}_2(x)H_2(x)=I_{\N(x)^\perp}$, by lemma~\ref{L:cyclic vector},  $\C(H_2(x))\neq \emptyset$.

(3)
$\bar{H}_1(x)\triangleq B_x\hat{T}_xB_x^{-1}\in B_1(\Phi_0)$ follows from Lemma~\ref{Theorem 1.5}.
\[
\bar{H}_1(x)\left(T^{-1}x\right)=B_x\hat{T}_x B_x^{-1}\left(T^{-1}x\right)=B_x\hat{T}_x\left(T^{-1}x\oplus(-e_0)\right)=0.
\]
Also
\[
\bar{H}_1(x)\bar{T}_1(x)=B_x\hat{T}_x B_x^{-1}B_x\hat{S}_x B_x^{-1}=I_{\N(x)}.
\]
Then $\bar{T}^*_1(x)\bar{H}^*_1(x)=I_{\N(x)}$. By Lemma~\ref{L:cyclic vector}, $\C(\bar{H}_1^*(x))\neq \emptyset$.

(4) Since $H_{21}(x)\bar{T}_1(x)=0$, $\bar{T}_1^*(x)H_{21}^*(x)=0$. So $H_{21}(x)$ is a rank one operator.
\end{proof}

\section{The main theorems}

In this section, it is helpful to keep the following special case in mind: $\sigma(T^{-1})$ is the union of two circles, one is with center $(0,0)$ and radius $2$ and the other is with center $\left(4,0\right)$ and radius $2$. In this special case
\[
\sigma(T^{-1})^\land=\{\lambda:\, |\lambda|\leq 2\}\cup \{\lambda:\, |\lambda-4|\leq 2\}
\]
and
\[
int \sigma(T^{-1})^\land=\{\lambda:\, |\lambda|< 2\}\cup \{\lambda:\, |\lambda-4|< 2\}.
\]

\begin{thm}\label{T:main theorem 1}
Suppose $T\in \B(\H)$ is an invertible operator, $x$ is a nonzero noncyclic vector of $T^{-1}$, and $\N(x)=span\{T^{-k}x:\, k\geq 1\}.$ Let $\overline{T}_1(x)=P_{\N(x)} T^{-1} P_{\N(x)}$. If $\sigma(\overline{T}_1(x))^\land\cap \rho_F(\overline{T}_1(x))$ has a connected component which does not contain zero point, then $T$ is intransitive.
\end{thm}
\begin{proof}
Suppose $T$ is transitive. By (1) of Proposition~\ref{P:main}, $\overline{T}_1^*(x)\in B_1(\Omega_0^*)$ and $0\in \Omega_0$. By Lemma~\ref{7.2}, let $\overline{T}_1^*(x)\in B_1(L_0^*)$ such that $L_0^*\supseteq \Omega_0^*$ is maximal. Suppose $\Omega$ is a connected component of $\sigma(\overline{T}_1(x))^\land\cap \rho_F(\overline{T}_1(x))$ such that $0\notin\Omega$. Then $\Omega\cap L_0=\emptyset$. By Proposition~\ref{P:main}, $T^{-1}x\in \mathcal{C}(\overline{T}_1(x))$. By Proposition~\ref{4.3}, $\Omega\subseteq \rho(\overline{T}_1(x))$. Note that $\partial \Omega\subseteq \sigma_e(\overline{T}_1(x))$. Since
$\overline{H}_1(x)\overline{T}_1(x)=I$ and $\overline{T}_1(x)^*\in B_1(\Omega_0^*)$, $\pi(\overline{H}_1(x))=\pi(\overline{T}_1(x))^{-1}$. Then $(\partial \Omega)^{-1}\subseteq \sigma_e(\overline{H}_1(x))$ and $\Omega^{-1}\subseteq \rho_F(\overline{H}_1(x))$. By (3) of Proposition~\ref{P:main}, we have $\overline{H}_1(x)\in B_1(\Phi_0)$ and $0\in\Phi_0$. By Lemma~\ref{7.2}, let $\overline{H}_1(x)\in B_1(L_1)$ such that $L_1\supseteq \Phi_0$ is maximal. Then $\Omega^{-1}\cap L_1=\emptyset$. By (3) of Proposition~\ref{P:main}, $\C(\overline{H}_1^*(x))\neq \emptyset$. By Proposition~\ref{4.3}, $\Omega^{-1}\subseteq \rho(\overline{H}_1(x))$. Note that $\overline{H}_1(x)\overline{T}_1(x)=I_{\N(x)}$. For $\lambda\in \Omega$, $\frac{1}{\lambda}\in \Omega^{-1}$ and
\[
\lambda\left(\overline{H}_1(x)-\frac{1}{\lambda}\right)=\lambda\left(\overline{H}_1(x)-\frac{1}{\lambda} \overline{H}_1(x)\overline{T}_1(x)\right)=\overline{H}_1(x)(\lambda-\overline{T}_1(x)).
\]
Note that both $\overline{H}_1(x)-\frac{1}{\lambda}$ and $\lambda-\overline{T}_1(x)$ are invertible. Thus $\overline{H}_1(x)$ is invertible. On the other hand $\overline{H}_1(x)\in B_1(\Phi_0)$ and $0\in \Phi_0$. So $\overline{H}_1(x)$ is not invertible. This is a contradiction.
\end{proof}

\begin{lem}\label{5.3}
 Suppose $T$ is transitive, $x$ is a nonzero noncyclic vector of $T^{-1}$, and $\mathcal{N}(x)=span\{T^{-k}x:\, k\geq 1\}$.  Let $\mathcal{U}_0$ denote the connected component of int$(\sigma(T^{-1})^{\land})$ containing $0$ and let $\Omega$ be a connected open subset of $\rho(T^{-1})$. If $\Omega\cap \mathcal{U}_0=\emptyset$, then $\Omega\subseteq\rho(\overline{T}_1(x))\cap \rho(\overline{T}_2(x))$.
\end{lem}
\begin{proof}
Recall that
\[
 T^{-1}=\begin{pmatrix}
 \overline{T}_1(x) &\quad \overline{T}_{12}(x) \\
 0 &\quad \overline{T}_{2}(x)\\
 \end{pmatrix}\begin{matrix}\mathcal{N}(x)&\\
\mathcal{N}(x)^{\bot}&\end{matrix}.
\]
By Proposition~\ref{P:main}, $\bar{T}_1^*(x)\in B_1(\Omega_0^*)$, $\bar{T}_2(x)\in B_1(\Omega_0)$ and $0\in\Omega_0$. By Lemma~\ref{3.4}, $$\sigma_p(\bar{T}_1(x))=\sigma_p(\bar{T}_2^*(x))=\emptyset.$$
Suppose $\lambda\in \rho(T^{-1})$. Write
\[
\left(T^{-1}-\lambda\right)^{-1}=\begin{pmatrix}
 A &\quad B \\
 C &\quad D\\
 \end{pmatrix}\begin{matrix}\mathcal{N}(x)&\\
\mathcal{N}(x)^{\bot}&\end{matrix}.
\]
Then we have
\[
\begin{pmatrix}
 A &\quad B \\
 C &\quad D\\
 \end{pmatrix}\begin{pmatrix}
 \overline{T}_1(x)-\lambda  &\quad \overline{T}_{12}(x) \\
 0 &\quad \overline{T}_{2}(x)-\lambda\\
 \end{pmatrix}=\begin{pmatrix}
 \overline{T}_1(x)-\lambda &\quad \overline{T}_{12}(x) \\
 0 &\quad \overline{T}_{2}(x)-\lambda\\
 \end{pmatrix}\begin{pmatrix}
 A &\quad B \\
 C &\quad D\\
 \end{pmatrix}=\begin{pmatrix}
 I_{\N(x)} &0 \\
 0 &I_{\N(x)^\perp}\\
 \end{pmatrix}.
\]
This implies that $(\overline{T}_{2}(x)-\lambda)D=I_{\N(x)^\perp}$ and $A(\overline{T}_1(x)-\lambda)=I_{\N(x)}$ or equivalently $$(\overline{T}_1(x)-\lambda)^*A^*=I_{\N(x)}.$$ Therefore, $Ran(\overline{T}_{2}(x)-\lambda)$ and $Ran(\overline{T}_{1}(x)-\lambda)^*$ are closed. Thus $Ran(\overline{T}_{1}(x)-\lambda)$ is also closed.  Since $$\sigma_p(\bar{T}_1(x))=\sigma_p(\bar{T}_2^*(x))=\emptyset,$$
$\lambda\in \rho_{s-F}(\overline{T}_{1}(x))\cap \rho_{s-F}(\overline{T}_{2}(x))$. We have
\[
\rho(T^{-1})\subseteq \rho_{s-F}(\overline{T}_{1}(x))\cap \rho_{s-F}(\overline{T}_{2}(x)).
\]
Since $\Omega\subseteq \rho(T^{-1})$, $\Omega\subseteq \rho_{s-F}(\overline{T}_{1}(x))\cap \rho_{s-F}(\overline{T}_{2}(x))$.
 Note that $\overline{T}_1^{*}(x)\in B_1(\Omega_{0}^*)$ and $0\in \Omega_0$.  By Lemma~\ref{7.2}, $\overline{T}_1^*(x)\in B_1(L_0^*)$ such that $L_0^*\supseteq \Omega_0^*$ is maximal. Note that $\mathcal{N}(x)\in Lat(T^{-1}).$ By Lemma~\ref{L:invariant subspace},  we have $\sigma(\overline{T}_1(x))\subseteq\sigma(T^{-1})^{\land}$ and thus  $\sigma(\overline{T}_1(x))^\land\subseteq \sigma(T^{-1})^\land$. Then the connected component of int$\sigma(\overline{T}_1(x))^\land$ containing zero point is a subset of $\U_0$. Since $L_0\subseteq \sigma(\overline{T}_1(x))$ is a connected open set containing zero point, $L_0\subseteq \U_0$. By the assumption of Lemma~\ref{5.3}, $\Omega\cap \U_0=\emptyset$. Thus $\Omega\cap L_0=\emptyset$. By Proposition~\ref{P:main}, $T^{-1}x\in \mathcal{C}(\overline{T}_1(x))$.
By Proposition \ref{4.3},
$\Omega\subset \rho(\overline{T}_1(x))$.

 Notice that $\Omega\subseteq \rho(T^{-1})\cap \rho(\overline{T}_1(x)).$   Therefore, $$0=\mbox{ind}(T^{-1}-\lambda)=\mbox{ind}(\overline{T}_1(x)-\lambda)+\mbox{ind}(\overline{T}_2(x)-\lambda)=0+\mbox{ind}(\overline{T}_2(x)-\lambda), \forall \lambda\in \Omega.$$
 Thus, we have $\mbox{ind}(\overline{T}_2(x)-\lambda)=0, \forall \lambda\in \Omega.$  By (2) of Proposition~\ref{P:main}, $\overline{T}_2(x)\in B_1(\Omega_0)$.
 By Lemma~\ref{3.4}, $\sigma_p(\overline{T}_2^*(x))=\emptyset$. Thus $dimker(\overline{T}_2(x)-\lambda)=0$, $\forall \lambda\in\Omega$, and $\Omega\subseteq\rho(\overline{T}_2(x))$.
\end{proof}

\begin{lem}\label{L:4 conditions}
Let $\sigma_e(T^{-1})=\sigma(T^{-1})\subseteq \sigma(\overline{T}_2(x))$ and let
\[
R_0=\left\{\frac{1}{\lambda}:\, \lambda\in\CCC\setminus \sigma(\overline{T}_1(x))^\land\cup\{\infty\}\right\},
\]
\[
\Phi_0=\left\{\frac{1}{\lambda}:\, \lambda\in\CCC\setminus \sigma(T^{-1})^\land\cup\{\infty\}\right\}.
\]
Then the following statements hold:
\begin{enumerate}
\item  $
\Phi_0=\left\{\frac{1}{\lambda}:\, \lambda\in\CCC\setminus \sigma(\overline{T}_2(x))^\land\cup\{\infty\}\right\};$
\item $\Phi_0\subseteq R_0$;
\item $\overline{H}_1(x)\in B_1(R_0)$, $R_0$ is maximal and $R_0^*$ is the connected component of $\rho_F^{-1}(\overline{H}_1^*(x_1))$ containing zero point.
\end{enumerate}
\end{lem}
\begin{proof}
(1). By the assumption of Lemma~\ref{L:4 conditions} and Lemma~\ref{L:invariant subspace},
\[
\sigma(T^{-1})\subseteq \sigma(\overline{T}_2(x))\subseteq \sigma(T^{-1})^\land.
\]
Then $\sigma(\overline{T}_2(x))^\land=\sigma(T^{-1})^\land$ and (1) holds.

(2). Since $\N(x)\in Lat (T^{-1})$, $\sigma(\overline{T}_1(x))\subseteq \sigma(T^{-1})^\land$. Therefore, $\sigma(\overline{T}_1(x))^\land\subseteq \sigma(T^{-1})^\land$. Then $\Phi_0\subseteq R_0$.

(3). By (3) of Proposition~\ref{P:main}, $\overline{H}_1(x)\in B_1(\Phi_0)$. Since $\overline{H}_1(x)\overline{T}_1(x)=I_{\N(x)}$, $$\pi(\overline{H}_1(x))=\pi(\overline{T}_1(x))^{-1}.$$   Since $$\CCC\setminus \sigma(\overline{T}_1(x))^\land\subset \rho(\overline{T}_1(x))\subset \rho(\pi(\overline{T}_1(x))),$$ $R_0\subseteq \rho(\pi(\overline{H}_1(x)))$. By (2), Lemma~\ref{3.4}, and continuity of index, $\overline{H}_1(x)\in B_1(R_0)$. Since $$\partial \sigma(\overline{T}_1(x)^\land)\subseteq \sigma_e(\overline{T}_1(x)),$$ we have $$\partial R_0\subseteq \sigma_e(\overline{H}_1(x)).$$ Thus
 $R_0^*$ is the connected component of $\rho_{F}^{-1}(\overline{H}_1^*(x))$ containing zero point.

\end{proof}

\begin{thm}\label{T:main theorem 2}
Let $T\in \B(\H)$ be an invertible operator. Suppose $x$ is a nonzero noncyclic vector of $T^{-1}$ and $\N(x)=span\{T^{-k}x:\,k\geq 1\}$. Let $\overline{T}_2(x)=P_{\N^\perp(x)}T^{-1}P_{\N^\perp(x)}$. If $\sigma(T^{-1})\subseteq \sigma(\overline{T}_2(x))$,  and there exists a bounded open set $\Omega$ which is a connected component of $\rho(T^{-1})$ such that $\Omega\cap \U_0=\emptyset$, where $\U_0$ is the connected component of $int(\sigma(T^{-1})^\land)$ containing zero point, then $T$ is intransitive.
\end{thm}
\begin{proof}

Suppose $T$ is transitive. Then $\sigma(T^{-1})=\sigma_e(T^{-1})$ and $\rho(T^{-1})=\rho_F(T^{-1})$.
 By Lemma~\ref{5.3}, since $\Omega\cap\U_0=\emptyset$,

(1.1) $\Omega\subseteq \rho(\overline{T}_1(x))\cap \rho(\overline{T}_2(x))$.

By Proposition~\ref{P:main}, $\overline{T}_1^*(x)\in B_1(\Omega_0^*)$ and $0\in \Omega_0$. By Lemma~\ref{7.2}, let $\overline{T}_1^*(x)\in B_1(L_0^*)$, and $L_0^*\supseteq \Omega_0^*$ is maximal. By Proposition~\ref{P:main}, $T^{-1}x\in\C(\overline{T}_1(x))$. By Proposition~\ref{4.3}, $L_0$ is the connected component of $\rho_F^{-1}(\overline{T}_1(x))$ containing zero point.

{\bf Claim 1:} $L_0\subseteq \U_0$. By Lemma~\ref{L:invariant subspace}, $\sigma(\overline{T}_1(x))\subseteq \sigma(T^{-1})^\land$. Thus
\[
L_0\subseteq int\sigma(\overline{T}_1(x))\subseteq int \sigma(T^{-1})^\land.
\]
Since $0\in L_0$, and $\U_0$ is the connected component of  $int(\sigma(T^{-1})^\land$ containing zero point, it follows that $L_0\subseteq \U_0$ and Claim 1 holds.

By the assumption of Theorem~\ref{T:main theorem 2},  $\Omega$ is a bounded connected component of $\rho(T^{-1})$. Since $\sigma(T^{-1})\subseteq\sigma(\overline{T}_2(x))$, $\rho(\overline{T}_2(x))\subseteq\rho(T^{-1})$.
 By (1.1), $\Omega\subseteq \rho(\overline{T}_2(x))\subseteq\rho(T^{-1})$ is a bounded connected component of $\rho(\overline{T}_2(x))$
 and $\partial \Omega\subseteq \sigma_e(\overline{T}_2(x))$. By Proposition~\ref{P:main}, $\overline{T}_2(x)\in B_1(\Omega_0)$ and $0\in \Omega_0$.
 Since $\overline{T}_2(x)H_2(x)=I_{\N(x)^\perp}$, $$\pi(\overline{T}_2(x))=\pi(H_2(x))^{-1}.$$  Thus $$\Omega^{-1}=\left\{\frac{1}{\lambda}:\,\lambda\in \Omega\right\}\subseteq \rho_F(H_2(x))$$ and $\partial \Omega^{-1}\subseteq \sigma_e(H_2(x))$. Since $\overline{T}_2(x)\in B_1(\Omega_0)$ and $0\in \Omega_0$, $ind(\overline{T}_2(x))=1$. Since $\overline{T}_2(x)H_2(x)=I$, $ind H_2(x)=-1$. Note that $0\notin \Omega^{-1}$ and $\C(H_2(x))\neq \emptyset$ by Proposition~\ref{P:main}. We have $\Omega^{-1}\cap \rho_F^{-1}(H_2(x))=\emptyset$. By Proposition~\ref{4.4},
\[
dimker(H_2(x)-\lambda)^*\leq 1,\quad\forall \lambda\in\CCC.
\]
Thus $ind(H_2(x)-\lambda)\geq -1$ for all $\lambda\in \CCC$. Since $\Omega^{-1}\cap \rho_F^{-1}(H_2(x))=\emptyset$,

(1.2) $ind(H_2(x)-\lambda)\geq 0,\quad\forall \lambda\in \Omega^{-1}$.

{\bf Claim 2.}\,  $\Omega^{-1}\subseteq R_0=\left\{\frac{1}{\lambda}:\, \lambda\in\CCC\setminus \sigma(\overline{T}_1(x))^\land\cup\{\infty\}\right\}$.

Note that either $\Omega^{-1}\cap R_0=\emptyset$ or $\Omega^{-1}\subseteq R_0$. In fact, since $\overline{H}_1(x)\overline{T}_1(x)=I_{\N(x)}$, $$\pi(\overline{H}_1(x))=\pi(\overline{T}_1(x))^{-1}.$$ By (1.1), $\Omega\subseteq \rho(\overline{T}_1(x))$. Thus $\Omega^{-1}\subseteq \rho_F(\overline{H}_1(x))$. By (3) of Lemma~\ref{L:4 conditions}, $\overline{H}_1(x)\in B_1(R_0)$ and $R_0$ is maximal. By Proposition~\ref{P:main}, $\C(\overline{H}_1^*(x))\neq \emptyset$. By Proposition~\ref{4.3}, $R_0=\rho_F^{1}(\overline{H}_1(x))$.
 If $\lambda_0,\lambda_1\in \Omega^{-1}$ and $\lambda_0\in R_0$, $\lambda_1\notin R_0$, then
\[
ind(\lambda_0-\overline{H}_1(x))=1
\]
and
\[
ind(\lambda_0-\overline{H}_1(x))\neq 1.
\]
This contradicts to the continuity of index.

Now suppose $\Omega^{-1}\cap R_0=\emptyset$.   By Proposition~\ref{P:main},
$\C(\overline{H}_1^*(x))\neq \emptyset$. By Proposition~\ref{4.3}, $\Omega^{-1}\subseteq\rho(\overline{H}_1(x))$. For $\lambda\in \Omega$,
\[
\overline{H}_1(x)(\lambda-\overline{T}_1(x))=\lambda\left(\overline{H}_1(x)-\frac{1}{\lambda}I\right).
\]
 By (1.1), both $\lambda-\overline{T}_1(x)$ and $\overline{H}_1(x)-\frac{1}{\lambda}I$ are invertible. So $\overline{H}_1(x)$ is invertible. By Proposition~\ref{P:main}, $\overline{H}_1(x)\in B_1(\Phi_0)$ and $0\in\Phi_0$. So $\overline{H}_1(x)$ is not invertible.
  This is a contradiction. Thus Claim 2 holds.

By Claim 2 and (3) of Lemma~\ref{L:4 conditions}, we have $ind (\overline{H}_1(x)-\lambda)=1$ for $\lambda\in \Omega^{-1}$. Since $\Omega\subseteq \rho(T^{-1})$, $\Omega^{-1}\subseteq \rho(T)$. By (1.2),
\[
0=ind (T-\lambda)=ind(\overline{H_1}(x)-\lambda)+ind(H_2(x)-\lambda)\geq 1+0=1,\quad\forall \lambda\in \Omega^{-1}.
\]
This is a contradiction.
\end{proof}

\section{Applications}

\subsection{ Hyponormal operators}

\begin{defn} Let $T\in \mathcal{B}(\mathcal{H})$. $T$ is  called hyponormal if $T^*T-TT^*\geq 0$.

\end{defn}

The following proposition is well-known.

\begin{prop} Let $T\in \mathcal{B}(\mathcal{H})$ be an invertible hyponormal operator and $\mathcal{N}\in Lat(T)$.  Then $T|_{\mathcal{N}}$ and $T^{-1}$ are both hypernormal.

\end{prop}

Suppose $A\in B_1(\Omega)$ and $0\in\Omega$. For $n\geq 1$, choose a unit vector $e_{n-1}\in ker A^n\ominus ker A^{n-1}$. Then $\{e_n\}_{n=0}^\infty$ is an ONB of $\H$ and
$$A=\begin{pmatrix}
0&\quad a_{01}&\quad a_{02}&\quad a_{03}&\quad \cdots\\
0&\quad 0&\quad a_{12}&\quad a_{13}&\quad\cdots\\
0&\quad  0&\quad 0&\quad a_{23}&\quad\cdots\\
\vdots&\quad\vdots&\quad\vdots&\quad\vdots&\quad\ddots
\end{pmatrix}\begin{matrix}
e_{0}\\
e_{1}\\
e_{2}\\
\vdots
\end{matrix}.
$$
Operator $B$ is called \emph{the standard right inverse} of $A$ if $AB=I$ and
$$
B=\begin{pmatrix}
0&\quad 0&\quad 0&\quad 0&\quad\cdots\\
b_{21}&\quad b_{22}&\quad b_{23}&\quad b_{24}&\quad\cdots\\
b_{31}&\quad b_{32}&\quad b_{33}&\quad b_{34}&\quad\cdots\\
\vdots&\quad\vdots&\quad\vdots&\quad\vdots&\quad\ddots
\end{pmatrix}\begin{matrix}
e_{0}\\
e_{1}\\
e_{2}\\
\vdots
\end{matrix}.$$

In order to prove the main result in this subsection, we need the following lemmas.

\begin{lem}\label{JWnew} Suppose that  $A\in B_1(\Omega)$, $0\in \Omega$ and $B$ is the standard right inverse of $A$. Write
$$A=\begin{pmatrix}
0&\quad a_{01}&\quad a_{02}&\quad a_{03}&\quad \cdots\\
0&\quad 0&\quad a_{12}&\quad a_{13}&\quad\cdots\\
0&\quad  0&\quad 0&\quad a_{23}&\quad\cdots\\
\vdots&\quad\vdots&\quad\vdots&\quad\vdots&\quad\ddots
\end{pmatrix}\begin{matrix}
e_{0}\\
e_{1}\\
e_{2}\\
\vdots
\end{matrix}.$$
Then $|a_{k, k+1}|\geq \frac{1}{\|B\|}.$
\end{lem}

\begin{proof} Since $B$ can be regarded as an invertible operator from $\mathcal{H}$ onto $\ker A^{\perp}$. So for any $y\in \ker A^{\perp}$,  we have $y=ABy=BAy.$  Then it follows that $\|y\|\leq \|B\|\cdot \|Ay\|$, and
$\|Ay\|\geq \frac{1}{\|B\|}\|y\|.$
Note that
\[
a_{0,1}e_0=Ae_1,
\]
\[
a_{1,2}e_1=A\left(e_2-\frac{a_{0,2}}{a_{0,1}}e_1\right),
\]
\[
a_{2,3}e_2=Ae_3-a_{0,3}e_0-a_{1,3}e_1=Ae_3-\frac{a_{0,3}}{a_{0,1}}Ae_1-\frac{a_{1,3}}{a_{1,2}}A\left(e_2-\frac{a_{0,2}}{a_{0,1}}e_1\right)
\]
\[
=A(e_3+y_3),
\]
where $y_3$ is in span$\{e_1,e_2\}$. By induction, we have
\[
a_{k,k+1}e_k=A(e_{k+1}+y_{k+1}),
\]
where $y_{k+1}$ is in span$\{e_1,\cdots,e_k\}$.
 Thus, we have that
$$\|a_{k,k+1}\|=\|A(e_{k+1}+y_{k+1})\|\geq \frac{1}{\|B\|}\|e_{k+1}+y_{k+1}\|=\frac{1}{\|B\|}(1+\|y_{k+1}\|^2)^{\frac{1}{2}}\geq \frac{1}{\|B\|}.$$

\end{proof}

\begin{lem}\label{7.4} Let $T\in \mathcal{B}(\mathcal{H})$ be invertible and transitive. Suppose that $x$ is a nonzero noncyclic vector of $T^{-1}$ and $\mathcal{N}(x)=\mbox{span}\{T^{-k}x, k\geq 1\}$.  Then
$$T=\begin{pmatrix}
 H_1(x) &\quad H_{12}(x) \\
 f(x)\otimes g(x) &\quad H_{2}(x)\\
\end{pmatrix}\begin{matrix}\mathcal{N}(x)&\\
\mathcal{N}(x)^{\bot}&\end{matrix},\ T^{-1}=\begin{pmatrix}
 \overline{T}_1(x) &\quad \overline{T}_{12}(x) \\
 0 &\quad \overline{T}_{2}(x)\\
 \end{pmatrix}\begin{matrix}\mathcal{N}(x)&\\
\mathcal{N}(x)^{\bot}&\end{matrix}. $$
Furthermore, there exists an ONB $\{u_k\}^{\infty}_{k=-\infty}$ of $\mathcal{H}$, $u_k\in ker \overline{T}^{*(k+1)}_1(x)\ominus ker\overline{T}^{*k}_1(x)$ for $k\geq 0$, $u_{-k}\in ker \overline{T}_2^k(x)\ominus ker\overline{T}_2^{k-1}(x)$ for $k\geq 1$, which satisfies
$$\overline{T}^*_1(x)=\begin{pmatrix}
0&\quad t_{0-1}&\quad t_{0-2}&\quad t_{0-3}&\quad \cdots\\
0&\quad 0&\quad t_{-1-2}&\quad t_{-1-3}&\quad \cdots\\
0&\quad 0&\quad 0&\quad t_{-2-3}&\quad \cdots\\
\vdots&\quad \vdots&\quad \vdots&\quad \vdots&\quad \ddots
\end{pmatrix}\begin{matrix}
u_{0}\\
u_{-1}\\
u_{-2}\\
\vdots
\end{matrix},$$ and $$ \overline{T}_2(x)=\begin{pmatrix}
0&\quad t_{12}&\quad t_{13}&\quad t_{14}&\quad \cdots\\
0&\quad 0&\quad t_{23}&\quad t_{24}&\quad \cdots\\
0&\quad 0&\quad 0&\quad t_{34}&\quad \cdots\\
\vdots&\quad \vdots&\quad \vdots&\quad \vdots&\quad \ddots
\end{pmatrix}\begin{matrix}
u_{1}\\
u_{2}\\
u_{3}\\
\vdots
\end{matrix},$$
where $|t_{-k,-k+1}|\geq \frac{1}{2\|T\|}, |t_{k, k+1}|\geq \frac{1}{2\|T\|}, k=0,1,2\cdots.$

\end{lem}

\begin{proof} Note that $\overline{T}^*_1(x)H^*_1(x)=I_{\mathcal{N}(x)}$. Write
\[
H^*_1(x)=\begin{pmatrix}
s_{01}&\quad s_{02}&\quad s_{03}&\quad s_{04}&\quad \cdots\\
s_{11}&\quad s_{12}&\quad s_{13}&\quad s_{14}&\quad \cdots\\
s_{21}&\quad s_{22}&\quad s_{23}&\quad s_{24}&\quad \cdots\\
\vdots&\quad \vdots&\quad \vdots&\quad \vdots&\quad \ddots
\end{pmatrix}\begin{matrix}
u_{0}\\
u_{-1}\\
u_{-2}\\
\vdots
\end{matrix}
\]
and
\[
B=\begin{pmatrix}
0&\quad 0&\quad 0&\quad 0&\quad \cdots\\
s_{11}&\quad s_{12}&\quad s_{13}&\quad s_{14}&\quad \cdots\\
s_{21}&\quad s_{22}&\quad s_{23}&\quad s_{24}&\quad \cdots\\
\vdots&\quad \vdots&\quad \vdots&\quad \vdots&\quad \ddots
\end{pmatrix}\begin{matrix}
u_{0}\\
u_{-1}\\
u_{-2}\\
\vdots
\end{matrix}.
\]
Then $\overline{T}^*_1(x)B=I_{\mathcal{N}(x)}$ and $B$ is the standard right inverse of $\overline{T}^*_1(x)$.
 By Lemma \ref{JWnew}, we have that  $$|t_{-k,-k+1}|\geq \frac{1}{\|B\|}\geq \frac{1}{2\|H^*_1(x)\|}\geq \frac{1}{2\|T\|}.$$
Notice that $T_2(x)H_2(x)=I_{\mathcal{N}(x)^\perp}$, we also have that $$|t_{k, k+1}|\geq \frac{1}{2\|H_2(x)\|}\geq \frac{1}{2\|T\|}, \quad k=0,1,2\cdots.$$

\end{proof}

\begin{lem}\label{keyL2} Suppose $T\in \mathcal{B}(\mathcal{H})$ is invertible and transitive, and there exists an ONB $\{u_k\}^{\infty}_{k=-\infty}$ of $\mathcal{H}$ given as in Lemma~\ref{7.4} such that
$$
T^{-1}=\begin{pmatrix}
\bar{T}_1(x)&\quad \bar{T}_{12}(x)\\
0&\quad \bar{T}_2(x)
\end{pmatrix}\begin{matrix}
\mathcal{N}(x)\\
\mathcal{N}(x)^\perp
\end{matrix}$$ $$=\begin{pmatrix}
\ddots\\
&0&\quad t_{-2-3}&\quad t_{-1-3}&\quad t_{0-3}&\quad \ast\\
&&\quad 0&\quad t_{-1-2}&\quad t_{0-2}&\quad \ast&&\quad \ast\\
&&&0&\quad t_{0-1}&\quad \ast\\
&&&&0&\quad a_{01}&\quad \ast&\quad \ast&\quad \ast&\quad \cdots\\
&&&&&0&\quad t_{12}&\quad t_{13}&\quad t_{14}&\quad \cdots\\
&&0&&&0&\quad 0&\quad t_{23}&\quad t_{24}&\quad \cdots\\
&&&&&0&\quad 0&\quad 0&\quad t_{34}&\quad \cdots\\
&&&&&\vdots&\quad \vdots&\quad \vdots&\quad \ddots&
\end{pmatrix}
\begin{matrix}
\vdots\\
u_{-3}\\
u_{-2}\\
u_{-1}\\
u_0\\
u_1\\
u_2\\
u_3\\
\vdots\\
\end{matrix}.
$$
Then $a_{01}\neq 0.$
\end{lem}

\begin{proof} Notice that $$\overline{T}^*_1(x)=\begin{pmatrix}
0&\quad \bar{t}_{0-1}&\quad \bar{t}_{0-2}&\quad \bar{t}_{0-3}&\quad \cdots\\
0&\quad 0&\quad \bar{t}_{-1-2}&\quad \bar{t}_{-1-3}&\quad \cdots\\
0&\quad 0&\quad 0&\quad \bar{t}_{-2-3}&\quad \cdots\\
\vdots&\quad \vdots&\quad \vdots&\quad \vdots&\quad \ddots
\end{pmatrix}\begin{matrix}
u_{0}\\
u_{-1}\\
u_{-2}\\
\vdots
\end{matrix}$$ and set $$A=\begin{pmatrix}
0&\quad \bar{t}_{-1-2}&\quad \bar{t}_{-1-3}&\quad \bar{t}_{-1-4}&\quad \cdots\\
0&\quad 0&\quad \bar{t}_{-2-3}&\quad \bar{t}_{-2-4}&\quad \cdots\\
0&\quad 0&\quad 0&\quad \bar{t}_{-3-4}&\quad \cdots\\
\vdots&\quad \vdots&\quad \vdots&\quad \vdots&\quad \ddots
\end{pmatrix}\begin{matrix}
u_{-1}\\
u_{-2}\\
u_{-3}\\
\vdots
\end{matrix}.$$
By  Lemma \ref{3.3}, there exists an invertible operator $G$ such that $$G\overline{T}^*_1(x)G^{-1}=A,\quad  G^{-1*}\overline{T}_1(x)G^*=A^*.$$  Since
$\mathcal{C}(\overline{T}_1(x))\neq \emptyset$,  we have $\mathcal{C}(A^*)\neq \emptyset$. Thus, there exists an $x_{-1}$ such that
$$\mathcal{N}(x_{-1})=\mbox{span}\{u_{-1},u_{-2},u_{-3}, \cdots\} \,\, \mbox{and}\,\,\overline{T}_1(x_{-1})=A^*.$$ Since $T$ is transitive, by Proposition~\ref{P:main}, we have that $\overline{T}_2(x_{-1})\in B_1(\Omega)$. If $\alpha_{01}=0$, we have $\mbox{dim}\ker \overline{T}_2(x_{-1})=2,$ which is a contradiction.

\end{proof}

\begin{lem}\label{7.7}Let $T\in \mathcal{B}(\mathcal{H})$ be invertible and transitive. Then there exists an $x_n\in \mathcal{H}$ such that $\mathcal{N}(x_n)=\mbox{span}\{u_k, k\leq n\},$ and satisfies
\begin{enumerate}
  \item [(1)]$\overline{T}_1(x_n)\sim_s \overline{T}_1(x_{n-1}) $;
  \item [(2)] $P_{\mathcal{N}(x_n)}\stackrel{SOT}{\longrightarrow} I$;
  \item[(3)] $P_{\mathcal{N}(x_n)}T^{-1}|_{\mathcal{N}(x_n)}\stackrel{SOT}{\longrightarrow} T^{-1}$.
    \end{enumerate}

\end{lem}

\begin{proof} Firstly, we will find $x_1$ such that $\mathcal{N}(x_1)=\mbox{span}\{u_k, k\leq 1\}.$ Set $$A_1=\begin{pmatrix}
0&\quad \bar{a}_{01}&\quad *&\quad *&\quad \cdots\\
0&\quad 0&\quad \bar{t}_{0-1}&\quad \bar{t}_{0-2}&\quad \cdots\\
0&\quad 0&\quad 0&\quad \bar{t}_{-1-2}&\quad \cdots\\
\vdots&\quad \vdots&\quad \vdots&\quad \vdots&\quad \ddots
\end{pmatrix}\begin{matrix}
u_{1}\\
u_{0}\\
u_{-1}\\
\vdots
\end{matrix}.$$
By Lemma \ref{keyL2}, we have  $a_{01}\neq 0$. By Lemma \ref{3.3}, $A_1\sim_s \overline{T}^*_1(x_0)$, where $x_0$ is the vector $x$ in Lemma \ref{keyL2}.  Since $T^{-1}x_0\in \mathcal{C}(\overline{T}_1(x_0))$, there exists some $x_1\neq 0$ such that $$\mathcal{N}(x_1)=\mbox{span}\{T^{-k}x_1, k\geq 1\}=\mbox{span}\{u_k, k\geq 1\}.$$
Meanwhile, $$T^{-1}|_{\mathcal{N}(x_1)}:=\overline{T}_1(x_1)=A^*_1,$$
and $T^{-1}x_1\in \mathcal{C}(\overline{T}_1(x_1))$. Now set
$$A_2=\begin{pmatrix}
0&\quad \bar{t}_{12}&\quad \bar{t}_{13}&\quad \bar{t}_{14}&\quad \cdots\\
0&\quad 0&\quad \bar{\alpha}_{01}&\quad *&\quad \cdots\\
0&\quad 0&\quad 0&\quad \bar{t}_{0-1}&\quad \cdots\\
\vdots&\quad \vdots&\quad \vdots&\quad \vdots&\quad \ddots
\end{pmatrix}\begin{matrix}
u_{2}\\
u_{1}\\
u_{0}\\
\vdots
\end{matrix}.$$
By Lemma \ref{7.4}, we have that $|t_{1,2}|\geq \frac{1}{2\|T\|}$. By Lemma \ref{3.3}, we have that $A_2\sim_s \overline{T}^*_{1}(x_1)$. Since $T^{-1}x_1\in \mathcal{C}(\overline{T}_1(x_1))$, there exists an $x_2$ such that
 $$\mathcal{N}(x_2)=\mbox{span}\{T^{-k}x_2, k\geq 1\}=\mbox{span}\{u_k, k\geq 2\}.$$
 In this case, set $\overline{T}_1(x_2)=A^*_2.$ Repeating the steps above,  we can find a sequence $\{x_n\}^{\infty}_{n=1}.$ Notice that $\mathcal{N}(x_{n-1})\subseteq \mathcal{N}(x_{n})$. Then $\{x_n\}^{\infty}_{n=1}$ satisfies the condition of this lemma.
\end{proof}

\begin{thm}\label{T:hypernormal}
Let $T\in \mathcal{B}(\mathcal{H})$ be an invertible hyponormal operator. If $T^{-1}$ is intransitive and there exist at least two bounded components of $\mbox{int} \sigma(T^{-1})^{\land}$, then $T$ is also intransitive.

\end{thm}

\begin{proof}
If $\mbox{int}\sigma(T^{-1})\neq \emptyset,$  by the result of S. Brown \cite{Brown}, we  see that
$T$ and $T^{-1}$ both have nontrivial invariants subspaces. So in the following we assume $\mbox{int}\sigma(T^{-1})=\emptyset.$

Let $x$ be a nonzero noncyclic vector of $T^{-1}$ and let $\N(x)=span\{T^{-n}x:\,n\geq 1\}$. Denote by $\overline{T}_1(x)=P_{\N(x)}T^{-1}P_{\N(x)}$.  Suppose $T$ is transitive. By Proposition~\ref{P:main}, $\overline{T}_1^*(x)\in B_1(\Omega_0^*)$  and $0\in \Omega_0$. By Lemma~\ref{7.2}, let $\overline{T}_1^*(x)\in B_1(L_0^*)$ such that $L_0^*\supseteq \Omega_0^*$ is maximal. By Proposition~\ref{P:main}, $\C(\overline{T}_1(x))\neq \emptyset$. By Proposition~\ref{4.3}, $L_0=\rho_F^{-1}(\overline{T}_1(x))$.

Let $\U_0$ and $\U_1$ be two bounded components of $\mbox{int} \sigma(T^{-1})^{\land}$ such that $0\in\U_0$. By the assumption $\mbox{int}\sigma(T^{-1})=\emptyset,$ we have $\U_0, \U_1\subseteq \rho(T^{-1})$. Since $\N(x)\in Lat (T^{-1})$, $\sigma(\overline{T}_1(x))\subseteq \sigma(T^{-1})^\land$ and $\sigma(\overline{T}_1(x))^\land\subseteq \sigma(T^{-1})^\land$. Therefore, $L_0\subseteq \U_0$ and $L_0\cap \U_1=\emptyset$.
 Suppose $\partial \U_1\subseteq \sigma_e(\overline{T}_1(x))$. Since $\U_1\subseteq \rho(T^{-1})$, $\U_1\subseteq\rho_{s-F}(\overline{T}_1(x))$. Note that $L_0\cap\U_1=\emptyset$ and $\C(\overline{T}_1(x))\neq \emptyset$. By Proposition~\ref{4.3}, $\U_1\subseteq \rho(\overline{T}_1(x))$ and $\sigma(\overline{T}_1(x))^\land\cap \rho_F(\overline{T}_1(x))$ has a connected component $\U_1$ which does not contain zero point. By Theorem~\ref{T:main theorem 1}, $T$ is intransitive. This contradicts to the assumption that $T$ is transitive. Thus there exists a $\lambda\in \partial \U_1$ such that $\lambda\in \rho_{F}(\overline{T}_1(x))$. Since $\C(\overline{T}_1(x))\neq \emptyset$, by Lemma~\ref{3.4} and Proposition~\ref{4.4}, $\lambda\in \rho_F^{-1}(\overline{T}_1(x))$ or $\lambda\in \rho(\overline{T}_1(x))$. Note that $L_0=\rho_F^{-1}(\overline{T}_1(x))$ and $L_0\cap  \U_1=\emptyset$.  We have $\lambda\notin \rho_F^{-1}(\overline{T}_1(x))$. Thus $\lambda\in \rho(\overline{T}_1(x))$.

By Lemma \ref{7.7}, there exists a sequence of vectors $\{x_n\}^{\infty}_{n=0}\in \mathcal{H}$ such that $\mathcal{N}(x_{n-1})\subseteq \mathcal{N}(x_n)$. Then $P_{\mathcal{N}(x_n)}\stackrel{SOT}{\longrightarrow} I$, $\overline{T}_1(x_n)\stackrel{SOT}{\longrightarrow} T^{-1}$ and $\overline{T}_1(x_n)\sim_s \overline{T}_1(x_{n-1})$ and thus $\sigma( \overline{T}_1(x_n))=\sigma(\overline{T}_1(x_0)), n\geq 0$. Furthermore, for any $\xi \in \mathcal{H}$, by [Proposition 2.1, page 72] of~\cite{MP}, we have
$$\begin{array}{lll}
\|(\overline{T}_1(x_n)-\lambda)P_{\mathcal{N}(x_n)}\xi\|&\geq & \frac{1}{\|(\overline{T}_1(x_n)-\lambda)^{-1}\|}\|P_{\mathcal{N}(x_n)}\xi\| \\
&=&dist\{\lambda, \sigma(\overline{T}_1(x_n))\}\|P_{\mathcal{N}(x_n)}\xi\| \\
&=&dist\{\lambda, \sigma(\overline{T}_1(x_0))\}\|P_{\mathcal{N}(x_n)}\xi\|
\end{array}$$
Notice that $$\|(T^{-1}-\lambda)\xi\|=\lim\limits_{n\rightarrow \infty}\|P_{\mathcal{N}(x_n)}(\overline{T}_1(x_n)-\lambda)P_{\mathcal{N}(x_n)}\xi\|=\lim\limits_{n\rightarrow \infty}\|(\overline{T}_1(x_n)-\lambda)P_{\mathcal{N}(x_n)}\xi\|.$$
It follows that $$\|(T^{-1}-\lambda)\xi\|\geq dist\{\lambda, \sigma(\overline{T}_1(x_0))\}\|\xi\|.$$
Therefore, $(T^{-1}-\lambda)$ is bounded below. On the other hand, $\lambda\in \partial \U_1\subseteq \partial\sigma(T^{-1})$. This is a contradiction.
\end{proof}

\subsection{Strictly cyclic invariant subspaces}

\begin{defn}
Let $T\in \B(\H)$ and let $\A(T)$ be the closed subalgebra of $\B(\H)$ generated by $T$ and $I$. A vector $\xi\in\H$ is called a strictly cyclic vector of $T$ if $\{S\xi:\, S\in\A(T)\}=\H$. An invariant subspace $\K\subseteq \H$ of $T$ is called a strictly cyclic invariant subspace if $T|_\K\in \B(\K)$ has a strictly cyclic vector.
\end{defn}

The following lemma is due to B.Barnes~(Corollary 3 of \cite{Bar}).
\begin{lem}\label{L:7.11}
Suppose $T\in \B(\H)$ is strictly cyclic. Then $\sigma_p(T^*)=\sigma(T^*)$.
\end{lem}

\begin{lem}\label{7.13} Let $T\in \mathcal{B}({\mathcal{H}})$ be an invertible operator. If $T^{-1}$ is intransitive and satisfies the following properties:
\begin{enumerate}
\item  there exists a bounded open set $\Omega$ which is a connected component of $\rho(T^{-1})$ such that $\Omega\cap \U_0=\emptyset$, where $\U_0$ is the connected component of $int(\sigma(T^{-1})^\land)$ containing zero point;
\item  $\sigma_p((P_{\mathcal{N}(x)}T^{-1}|_{\mathcal{N}(x)})^*)=\sigma((P_{\mathcal{N}(x)}T^{-1}|_{\mathcal{N}(x)})^*),$
\end{enumerate}
then $T$ is also intransitive.

\end{lem}

\begin{proof}
Suppose $T$ is transitive.

{\bf Claim:}\, $\sigma(T^{-1})\subseteq\sigma(\overline{T}_2(x))$. Otherwise, there exists a $\lambda\in \sigma(T^{-1})$ such that $\lambda\in \rho(\overline{T}_2(x))$. If $\lambda\in \rho(\overline{T}_1(x))$, then direct computation shows that $\lambda\in \rho(T^{-1})$. It is a contradiction. Suppose $\lambda\in \sigma(\overline{T}_1(x))$. Then $\bar{\lambda}\in \sigma((\overline{T}_1(x))^*)=\sigma_p((\overline{T}_1(x))^*)$. Thus there exists a nonzero vector $\xi\in \N(x)$ such that $(\overline{T}_1(x)-\lambda)^*\xi=0$. Note that $(\overline{T}_2(x)-\lambda)^*$ is invertible. There exists a vector $\eta\in \N(x)^\perp$ such that
\[
(\overline{T}_2(x)-\lambda)^*\eta=-\overline{T}_{12}(x)^*\xi.
\]
This implies that
\[
\left(T^{-1}-\lambda\right)^*\begin{pmatrix}
\xi\\
\eta
\end{pmatrix}=\begin{pmatrix}
(\overline{T}_1(x)-\lambda)^*&\quad 0\\
\overline{T}_{12}(x)^*&\quad (\overline{T}_2(x)-\lambda)^*
\end{pmatrix}\begin{pmatrix}
\xi\\
\eta
\end{pmatrix}=0.
\]
Thus $\sigma_p(T^{-1*})\neq \emptyset$ and $T$ is intransitive. It is a contradiction and the Claim is true. By Theorem~\ref{T:main theorem 2}, $T$ is intransitive.
\end{proof}

\begin{thm}\label{T:cyclic}
Let $T\in \B(\H)$ be invertible.  If $T^{-1}$ has a proper strictly cyclic invariant subspace and there exists a bounded open set $\Omega$ which is a connected component of $\rho(T^{-1})$ such that $\Omega\cap \U_0=\emptyset$, where $\U_0$ is the connected component of $int(\sigma(T^{-1})^\land)$ containing zero point, then $T$ is intransitive.
\end{thm}
\begin{proof}
Suppose $\N(x)$ is a strictly cyclic invariant subspace of $T^{-1}$. By Lemma~\ref{L:7.11}, $\sigma_p(\bar{T}_1^*(x))=\sigma(\bar{T}_1^*(x))$. By Lemma~\ref{7.13}, $T$ has a nontrivial invariant subspace.
\end{proof}

\end{document}